\documentclass[11pt]{article}
\usepackage{enumerate,enumitem}

\usepackage[margin=2.5cm]{geometry}

\usepackage[usenames,dvipsnames]{color}
\usepackage{mathrsfs} 
\usepackage{amsmath}
\usepackage{amsfonts}
\usepackage{amssymb}
\usepackage[english]{babel}
\usepackage{graphicx,epsf,cite,esint}
\usepackage{amsthm}
\usepackage{mathtools}
\usepackage{textcomp}

\usepackage{hyperref}
\hypersetup{colorlinks=true, pdfstartview=FitV, linkcolor=BrickRed,citecolor=BrickRed, urlcolor=BrickRed}

\usepackage{mathtools}
\mathtoolsset{showonlyrefs=true}

\usepackage{accents}
\usepackage[T1]{fontenc}
\usepackage[utf8]{inputenc}

\usepackage [autostyle, english = american]{csquotes}
\MakeOuterQuote{"}
\numberwithin{equation}{section}
\theoremstyle{plain}
\newtheorem*{theorem*}{Theorem}
\newtheorem*{lemma*}{Lemma}
\newtheorem{theorem}{Theorem}
\newtheorem{lemma}{Lemma}[section]
\newtheorem{corollary}[lemma]{Corollary}

\newtheorem{proposition}[lemma]{Proposition}

\theoremstyle{definition}

\newtheorem{remark}[lemma]{Remark}

\def\l{\lambda}

\def\T{{\bf T}}

\newcommand{\bbN}{{\mathbb{N}}}

\newcommand{\eps}{\varepsilon}

\providecommand{\ip}[1]{\langle#1\rangle}
\providecommand{\abs}[1]{\left\lvert#1\right\rvert}
\providecommand{\norm}[1]{\left\|#1\right\|}

\def\eps{\varepsilon}
\def\e{{\rm e}}
\def\dd{{\rm d}}

\def\ddt{{\frac{\dd}{\dd t}}}
\def\R {\mathbb{R}}

\def \l {\langle}
\def \r {\rangle}
\def\T {{\mathbb T}}
\def\de{{\partial}}
\newcommand{\D}{\mathcal{D}}
\newcommand{\N}{\mathbb{N}}

\newcommand{\Lin}{\mathcal{L}}
\newcommand{\PP}{\mathbb{P}}

\begin{document}

\title{\vspace*{-1cm}Stationary Structures near the Kolmogorov and Poiseuille Flows in the 2d Euler Equations}
\author{Michele Coti Zelati\thanks{Department of Mathematics, Imperial College London, London, SW7 2AZ, UK, \nolinkurl{m.coti-zelati@imperial.ac.uk}}, Tarek M.\ Elgindi\thanks{Mathematics Department, Duke University,
Durham, NC 27708, USA, \nolinkurl{tarek.elgindi@duke.edu}}, Klaus Widmayer\thanks{Institute of Mathematics, EPFL, Station 8, 1015 Lausanne, Switzerland, \nolinkurl{klaus.widmayer@epfl.ch}}}

\date{\today}
\maketitle

\vspace*{-.9cm}
\begin{abstract}
We study the behavior of solutions to the incompressible $2d$ Euler equations near two canonical shear flows with critical points, the Kolmogorov and Poiseuille flows, with consequences for the associated Navier-Stokes problems.

We exhibit a large family of new, non-trivial stationary states of analytic regularity, that are arbitrarily close to the Kolmogorov flow on the square torus $\mathbb{T}^2$. This situation contrasts strongly with the setting of some monotone shear flows, such as the Couette flow: in both cases the linearized problem exhibits an ``inviscid damping'' mechanism that leads to relaxation of perturbations of the base flows back to nearby shear flows. While this effect persists nonlinearly for suitably small and regular perturbations of some monotone shear flows, for the Kolmogorov flow our result shows that this is not possible. 

Our construction of these stationary states builds on a degeneracy in the global structure of the Kolmogorov flow on $\mathbb{T}^2$. In this regard both the Kolmogorov flow on a rectangular torus and the Poiseuille flow in a channel are very different, and we show that the only stationary states near them must indeed be shears, even in relatively low regularity $H^3$ resp.\ $H^{5+}$.

In addition, we show that this behavior is mirrored closely in the related Navier-Stokes settings: the linearized problems near the Poiseuille and Kolmogorov flows both exhibit an enhanced rate of dissipation. Previous work by us and others shows that this effect survives in the full, nonlinear problem near the Poiseuille flow and near the Kolmogorov flow on rectangular tori, provided that the perturbations lie below a certain threshold. However, we show here that the corresponding result cannot hold near the Kolmogorov flow on $\T^2$.
\end{abstract}

\setcounter{tocdepth}{2}
{\small \tableofcontents}

\section{Introduction}
Solutions to the incompressible Euler equations are notoriously difficult to understand: they exhibit few conserved quantities, lack of compactness, chaotic behavior, and many other mathematical challenges. Toward gaining a deeper qualitative understanding of their long time dynamics, it is natural to first investigate possible ``end state'' configurations, such as stationary states or time-periodic solutions. Such structures can play a central role in the evolution of a flow, and can even become dominant aspects of it. Moreover, in some such settings there are clear links to associated dynamics in the Navier-Stokes equations.

This work is devoted to such questions in the setting of the $2d$ Euler ($\nu=0$) or Navier-Stokes ($\nu>0$) equations
\begin{align}\label{eq:NSEvel}
\begin{cases}
\de_t U+(U\cdot \nabla) U+\nabla P=\nu\Delta U,\\
\nabla\cdot U=0.
\end{cases}
\end{align}
For a domain $D\subset\R^2$ with suitable boundary conditions, these equations describe a flow through its velocity field $U=(U_1,U_2):D\times\R\to\R^2$, with $P:D\to\R$ being the internal pressure and $\nu\geq 0$ the kinematic viscosity (inversely proportional to the Reynolds number). In two dimensions it is advantageous to work instead with the scalar vorticity $\Omega:=\de_x U_2-\de_y U_1:D\times\R\to \R$, which satisfies
\begin{align}\label{eq:NSEvort}
\begin{cases}
 \de_t \Omega+U\cdot \nabla \Omega=\nu\Delta \Omega,\\
 U=\nabla^\perp \Psi, \quad \Delta\Psi=\Omega,
\end{cases}
\end{align}
and from which the so-called stream function $\Psi:D\times\R\to\R$ and the (divergence-free) velocity field $U$ of the flow can be recovered as described in \eqref{eq:NSEvort}. In this work we will be dealing chiefly with bounded, rectangular domains with (partially) periodic or Dirichlet boundary conditions, such as the square torus $\T^2$, a rectangular torus $\T^2_\delta:=[0,2\pi \delta]\times[0,2\pi]$, $\delta>0$ with $\delta\not\in\N$, or a channel $\T\times I$, where $I\subset\R$ is an interval.

While these problems are globally well-posed for sufficiently regular initial data, the long time behavior of their solutions is very hard to understand, especially in the case of the Euler equations.

\paragraph{Stationary States.}
A particularly important class of solutions to the Euler equations \eqref{eq:NSEvort} (with $\nu=0$) is given by stationary states, i.e.\ time-independent flow configurations. Their stream functions satisfy the equation
\begin{equation}
 \nabla^\perp\Psi\cdot\nabla\Delta\Psi=0,
\end{equation}
which holds in particular for solutions of the equation $\Delta\Psi=F(\Psi)$, for some $F\in C^1$. Two canonical solutions of this type are \emph{shear flows}\footnote{In a broader context and for more general geometries, such flows are also referred to as \emph{laminar flows}.}, where $\Psi$ depends on only one of the two spatial variables, and \emph{eigenfunctions of the Laplacian}.

Since the foundational investigations of Kelvin \cite{Kelvin87} and Reynolds \cite{Reynolds83} in the 1880's, shear flows have been important in both fluid dynamics theory and applications, and are commonly viewed as the natural state of a fluid in non-turbulent situations. On the other hand, eigenfunctions of the Laplacian are of interest, since the first non-trivial eigenfunctions maximize $\norm{U}_{L^2}$ for fixed $\norm{\Omega}_{L^2}$, which gives a natural stability mechanism. 
Due to the additional presence of viscosity, all solutions to the $2d$ Navier-Stokes equations \eqref{eq:NSEvort} (with $\nu>0$) are damped and eventually tend to zero. 
However, as we shall see later, in some cases there are close connections between stationary states of the Euler equations and certain special solutions of the Navier-Stokes equations.

A fundamental question in all these settings is how solutions near such stationary states behave, and in particular whether they are stable in a suitable sense, which has to be very carefully defined. Historically and until now, a natural starting point  
has been the investigation of ``modal stability'', i.e.\ the stability properties of the linearization near a given stationary state. This has uncovered two crucial effects due to vorticity mixing: in the Euler equations, so-called \emph{inviscid damping} is a mechanism that leads to damping of a component of the velocity \cite{BGM19}, whereas in the Navier-Stokes equations \emph{enhanced dissipation} produces an effective relaxation rate that is much faster than the natural diffusive one \cite{CKRZ08,CZDE18}.
These questions have received a lot of attention recently and have seen an enormous amount of progress, in particular for the case of shear flows \cite{BCZ17,GNRS20,LM19,WZ19,WZZ20,SMM19,CZEW19,Zillinger16} and vortices \cite{Gallay18,LWZ17}. In this context, the classic example is that of the Couette flow $U_\ast(y)=(y,0)$ on $\T\times\R$, which solves both Euler and Navier-Stokes equations, and was already investigated by Kelvin \cite{Kelvin87}. The linearized problem near $U_\ast$ can be solved explicitly, and demonstrates clearly the mixing effects mentioned above -- see also the review paper \cite{BGM19} and references therein. 

The associated nonlinear problems, however, are substantially harder.  In the inviscid case ($\nu=0$), nonlinear asymptotic stability remained unresolved until the groundbreaking work of Bedrossian and Masmoudi \cite{BM15} for the Couette flow. They established that sufficiently regular and small perturbations converge strongly in $L^2$ (of velocity) to a shear flow \emph{near} $U_*$ as $t\rightarrow\infty$. As was shown in \cite{DM18}, the Gevrey regularity here is a crucial ingredient of the proof. The work \cite{BM15} has led to many subsequent results. The only nonlinear results on the $2d$ Euler equations that we are aware of are \cite{MZ20,IJ20,IJ19}, where the method of \cite{BM15} is extended (in a highly non-trivial way) to handle the case of \emph{monotone} shear flows on $\T\times [a,b]$. 
When $\nu>0$ experimental predictions and simulations for the Navier-Stokes equations near $U_\ast$ were confirmed mathematically: it was shown that there exists a certain threshold for the size of the initial data, below which enhanced dissipation also holds in the nonlinear viscous problem near $U_\ast$ \cite{BVW18,CLWZ18,BMV16}, provided the Reynolds number is large enough. 

In the present work we venture into unexplored directions where the natural analogues and generalizations of the aforementioned results \emph{do not} hold. In fact, we show that the basic picture of viewing the nonlinear problem as a suitable perturbation of the corresponding linear setting can break down completely.

\subsection{Main Results}
The mixing mechanism upon which the above works are based, discovered first by Orr \cite{Orr07}, relies heavily on the monotonicity of the base profile $U_\ast$. 
Once one leaves the realm of monotonic flows, two canonical flows come to mind: the \emph{Kolmogorov} and \emph{Poiseuille} flows $U_K$ and $U_P$, respectively, given by
\begin{equation}
\begin{alignedat}{2}
 &U_{K}(y)=(\sin(y),0),\qquad &&(x,y)\in \T^2 \textnormal{ or }(x,y)\in \T^2_\delta,\\
 &U_{P}(y)=(y^2,0),\qquad &&(x,y)\in \T\times I,
\end{alignedat}
\end{equation}
where $I\subset\R$ is an interval. Both are stationary solutions to the Euler equations, and $U_P$ also solves the full Navier-Stokes equations, whereas $U_K$ evolves as a so-called \emph{bar state} $U_{bar}(t,y):=\e^{-\nu t}U_K(y)$ when $\nu>0$.

In comparison with the Couette flow, they are both \emph{locally degenerate} in the sense that they have a critical point. While the two share this similarity, it turns out that $U_K$ on $\T^2$ also has a sort of \emph{global degeneracy}. This is closely tied to the setting of the square torus $\T^2$ (rather than a rectangular torus $\T^2_\delta$) and, as we show in our main results below, makes for a crucial difference: while the behavior of solutions near $U_P$ on $\T\times I$ or $U_K$ on $\T^2_\delta$ may have similarities to that near the Couette flow as in \cite{LZ11}, the situation near $U_K$ \emph{on the square torus} $\T^2$ is entirely different.
In concrete terms, the degeneracy of the global structure of $U_K$ on $\T^2$ implies that the linearized operator $\Lin_K$ near $U_K$ has a ``large'' kernel, which includes not only shears (as is natural for linearized operators near shear flows), but also two eigenfunctions of the Laplacian. This is a well-known fact, but still allows for linear inviscid damping and linear enhanced dissipation results \cite{WZ19,WZZ20,SMM19}, which demonstrate these effects in a precise, quantified fashion \emph{away from the kernel} of $\Lin_K$. 
However, the present work shows that these effects do not persist in the nonlinear problem. 

\paragraph{In the Euler equations.}\label{ssec:results-E}
Building on the global degeneracy of $U_K$ on $\T^2$, we construct a large family of new, non-trivial stationary states of analytic regularity, that are arbitrarily close to $U_K$ and do \emph{not} lie in the kernel of the linearized operator $\Lin_K:=\sin(y)(1+\Delta^{-1})\de_x$. Our result constructs the corresponding stream functions as perturbations of the stream function $\cos(y)$ of the Kolmogorov flow $U_K$.
\begin{theorem}[Stationary states near Kolmogorov]\label{thm:main_Kolmo}
 There exists $\eps_0>0$ such that for any $0<\eps\leq\eps_0$ there exist analytic functions $\Psi_\eps\in C^\omega(\T^2)$ and $F_\eps\in C^\omega(\R)$ satisfying
 \begin{equation}\label{eq:stat_sol}
  \Delta\Psi_\eps=F_\eps(\Psi_\eps)
 \end{equation}
 and\footnote{for the precise Gevrey-$1$ regularity statement see Proposition \ref{prop:elliptic}}
 \begin{equation}
  \norm{\cos(y)-\Psi_\eps}_{C^\omega(\T^2)}= O(\eps),
 \end{equation}
 with
 \begin{equation}\label{eq:stat_sol_nontriv}
  \ip{\Psi_\eps,\cos(x)\cos(4y)}=-\eps^2\frac{\pi^2}{128}+O(\eps^3).
 \end{equation}
 This shows that, \emph{arbitrarily close to the Kolmogorov flow}, there are families of non-trivial (i.e.\ not in the kernel of $\Lin_K$), non-shear and stationary solutions $U_\eps:=\nabla^\perp\Psi_\eps:\T^2\to\R^2$ of the incompressible Euler equations.
\end{theorem}

This is all the more remarkable, since as initial data in the linearized, inviscid problem, these solutions would experience the inviscid damping effect as in \cite{WZZ20}, but are in fact simply stationary for the full Euler equations! Since moreover our stationary states are analytic, this shows that the linear inviscid damping results of \cite{WZZ20} cannot be extended in a perturbative spirit to the nonlinear setting, \emph{no matter the regularity}. This is in striking contrast with the case of the Couette flow \cite{LZ11}: there, similar stationary structures can only exist at low regularity ($H^{7/2}$ for the stream function) and as discussed earlier, nonlinear inviscid damping holds at sufficiently high regularity.

\begin{proof}[About the proof of Theorem \ref{thm:main_Kolmo} (full details in Section \ref{sec:Kolmo})]
 Our construction builds perturbatively on the fact that the stream function $\Psi_K:=\cos(y)$ of the Kolmogorov flow $U_K$ satisfies
 \begin{equation}\label{eq:psiKolmo}
  \Delta \Psi_K=F_K(\Psi_K),\qquad F_K(z)=-z.
 \end{equation}
 To find a larger class of solutions to \eqref{eq:stat_sol}, we make the ansatz 
 \begin{equation}
  \Psi_\eps=\Psi_K+\eps\psi,\qquad F_\eps=F_K+\eps f,
 \end{equation}
 which yields a nonlinear elliptic equation for $\psi$, with $f$ to be determined as well,
  \begin{equation}\label{eq:psi_elliptic_general}
  \Delta \psi+\psi=f(\Psi_K+\eps\psi).
 \end{equation}
 
 Notice that here a crucial difference with previous works \cite{CS12,CDG20} is that the operator $\Delta+1$ on the left hand side of \eqref{eq:psi_elliptic_general} is not invertible.
 This global degeneracy thus leads to some complications, but also allows us to introduce here via $\psi$ elements of the kernel $\ker\Lin_K$. Via the nonlinear interaction, this produces a plethora of different modes, and in particular allows for a construction of $\Psi_\eps$ such that the resulting flow is not inside the kernel of the linearization $\Lin_K$.
 
 At a more technical level, our proof constructs in tandem both the solution $\psi$ and the nonlinearity $f$ via a contraction argument. This is first done for $\psi\in H^2(\T^2)$, and it turns out that a simple choice for $f$ works well: that of an odd, real quintic polynomial (the coefficients of which are part of the contraction argument).\footnote{Other choices of $f$ are certainly possible, and one sees easily that in fact we may construct many different families of solutions to \eqref{eq:stat_sol} -- see also Remark \ref{rem:KolmoH2} in Section \ref{sec:Kolmo}.}
Given the relatively explicit nature of our construction, one can then easily find an expansion of $\Psi_\eps$, from which \eqref{eq:stat_sol_nontriv} follows directly.
 
 Finally, the analytic regularity can be deduced from \eqref{eq:psi_elliptic_general} via an elliptic regularity argument, which we detail in Section \ref{ssec:elliptic}. This also yields uniform in $\eps>0$ analytic Gevrey-$1$ norm bounds.
\end{proof}

\begin{remark}
\begin{enumerate}
 \item As this proof shows, a relatively simple form of $F_\eps$ as an odd, quintic polynomial suffices. Moreover, one can easily modify our arguments to show that many families $(\Psi_\eps)_\eps$ as in Theorem \ref{thm:main_Kolmo} exist -- see also Remark \ref{rem:KolmoH2} in Section \ref{sec:Kolmo} below.
 \item Theorem \ref{thm:main_Kolmo} also implies the existence of stationary states near the Kolmogorov flow on general tori $\T^2_N:=[0,2\pi N]\times[0,2\pi]$, $N\in\N$, with integer side length ratio, since we may simply embed $N$ copies of $\T^2$ in such a torus $\T^2_N$.
\end{enumerate}
\end{remark}

From the discussion so far one may be tempted to conjecture that one could find stationary states of the $2d$ Euler equations on $\T^2$ near Kolmogorov, which depart in any direction in the kernel of the linearization $\Lin_K$. However, we establish that this is not the case: we show in Section \ref{ssec:Kolmo_obstruct} that there are elements of $\ker\Lin_K$ which cannot arise as projections of stationary states.
\begin{proposition}\label{prop:main_someflex}
 There exists an infinite-dimensional linear subspace $Y\subset\ker\Lin_K$ such that if a solution $\Omega$ to the $2d$ Euler equations satisfies that $\norm{\Omega-\cos(y)}_{H^6}$ is sufficiently small and its projection $\PP_K(\Omega-\cos(y))$ onto $\ker\Lin_K$ satisfies $\PP_K(\Omega-\cos(y))\in Y$, then $\Omega$ cannot be stationary.
\end{proposition}

Highlighting the role of the global (versus local) degeneracy, we show that a similar construction of stationary states as in our Theorem \ref{thm:main_Kolmo} is \emph{not} possible near the Kolmogorov flow on a rectangular torus or near the Poiseuille flow in a channel. In fact, we show that all nearby stationary states must simply be shear flows, even in relatively low regularity $H^3$ and $H^{5+}$, respectively.
\begin{theorem}[Rigidity near Kolmogorov on a rectangular torus]\label{thm:main_rigid_intro}
 Consider the stationary solution $U_K(x,y)=(\sin(y),0)$ on $\T^2_{\delta}$, $\delta>0$ with $\delta\not\in\N$, of the Euler equations \eqref{eq:NSEvel}. There exists $\eps_0>0$ (depending on $\delta$) such that if $U:\T^2_\delta\to\R^2$ is a further stationary solution to the Euler equations with 
 \begin{equation}
  \norm{U-U_K}_{H^3}\leq \eps_0,
 \end{equation}
 then $U=U(y)$ is necessarily a shear flow.
\end{theorem}

Note that this rigidity does not only hold in the range $0<\delta<1$, but extends even for tori $\T^2_\delta$ with $\delta>1$, as long as $\delta\not\in\N$. This is remarkable, as in those settings the Kolmogorov flow has been proven to be linearly unstable \cite{MS61,FH98,FSV97}. 

\begin{proof}[About the proof of Theorem \ref{thm:main_rigid_intro} (full details in Section \ref{ssec:Kolmo_rigid})]
 Our proof builds on the fact that (in contrast to the setting on $\T^2$) the linearization $\Lin_K$ near $U_K$ on $\T^2_\delta$ only has shears in its kernel. From this we derive a coercivity estimate for nearby solutions, that allows them to only be shears, provided they are sufficiently close to $U_K$.
\end{proof}

In the setting of the Poiseuille flow, we demonstrate the stronger result that even any nearby travelling wave solution must simply be a shear flow.
\begin{theorem}[Rigidity near Poiseuille]\label{thm:main_Pois}
Let $s>5$, and consider the $2d$ Euler equations on $\T\times[-1,1]$
\begin{equation}\label{eq:Echannel}
 \de_t U+U\cdot\nabla U+\nabla P=0,\qquad \nabla\cdot U=0,\qquad U_2(x,\pm 1)=0.
\end{equation}
There exists $\eps_0>0$ such that if $U(x-ct,y)$, with $c\in \R$, is any traveling wave solution to \eqref{eq:Echannel} that satisfies 
\begin{align}\label{eq:poisClose}
 \norm{\Omega+2y}_{H^s}\leq\eps_0, \quad \textnormal{where }U=\nabla^\perp \Psi, \quad \Delta \Psi =\Omega,
\end{align}
then it follows that $U\equiv(U_1,0)$, that is, $U$ is necessarily a shear flow.
\end{theorem}

\begin{proof}[About the proof of Theorem \ref{thm:main_Pois} (full details in Section \ref{sec:Pois})]
 The proof of this result relies on a strong coercivity estimate for linearized operators around shears that are themselves close to the Poiseuille flow in the Euler equations. This further illuminates the contrast with the setting of the Kolmogorov flow, where no such estimate for the linearized operator $\Lin_K$ is available on $\T^2$.
 Combining this coercivity bound with the equations satisfied near $U_P$, we then obtain a contradiction if $U$ is both non-shear and close to $U_P$, as in the statement of Theorem \ref{thm:main_Pois}.
\end{proof}

\paragraph{In the Navier-Stokes equations.}\label{ssec:results-NS}
The above behavior is closely mirrored in the related Navier-Stokes settings: the linearized problems near the Poiseuille flow and the bar states (connected to the Kolmogorov flow) both exhibit an enhanced rate of dissipation \cite{CZEW19,WZZ20,WZ19}. Already early experiments of Reynolds on pipe flows \cite{Reynolds83} showed that such effects cannot be expected to occur in the nonlinear setting in general. Instead, one may hope to establish the existence of a nonlinear stability threshold depending on characteristic quantities of the flow (the so-called Reynolds number, here inversely proportional to the kinematic viscosity $\nu>0$): for initial data below the threshold, the nonlinear problem can be treated perturbatively and linear effects persist, whereas above it turbulent motion and instabilities may occur. And indeed, results of this type have been first demonstrated for monotone shear flows, with subsequent refinements on the precise size of the threshold \cite{BVW18,MZcrit19,MZco19,CLWZ18}

Our previous work \cite{CZEW19} proved the existence of such a threshold near the Poiseuille flow in the Navier-Stokes equations, while for the bar states on rectangular tori $\T^2_\delta$ with $0<\delta<1$ this was shown in \cite{WZZ20}. In stark contrast to these results, we show here that \emph{the corresponding result cannot hold for the bar states on $\T^2$}.

To make this precise, let us define the space $\D:=(\ker\Lin_K)^\perp$ and denote by $\PP_\D$ the associated orthogonal projection onto $\D$. In vorticity formulation, the linearization of the Navier-Stokes equations near the bar states $\Omega_{bar}=-\e^{-\nu t}\cos(y)$ is then given by
\begin{equation}\label{eq:Bar_lin}
 \de_t f+\e^{-\nu t}\Lin_K f=\nu\Delta f.
\end{equation}
The results of \cite{WZ19,WZZ20,SMM19} show that the enhanced dissipation in this linearized setting can be quantitatively captured by the statement that solutions $f(t)$ to \eqref{eq:Bar_lin} satisfy
\begin{equation}
 \norm{\PP_\D f(t)}_{L^2}\lesssim \e^{-c_1\nu^{1/2}t}\norm{\PP_\D f(0)}_{L^2}, \qquad \forall t\leq \frac{\tau}{\nu},
\end{equation}
where $c_1>0$ is some universal constant and $\tau>0$ is arbitrary.

Our next result demonstrates that there cannot be any threshold below which this $L^2$ decay also holds in the nonlinear Navier-Stokes problem near the bar states on $\T^2$, since there exist initial data arbitrarily close to those of the bar states that do not decay before the diffusive time scale $O(\nu^{-1})$ is reached.
\begin{theorem}\label{thm:main_Bar}
 For any $\nu>0$ there exists $0<\eps_0\ll \nu$ with the following property: let 
 $0<\eps\leq\eps_0$ and let $\Omega_\eps=\Delta\Psi_\eps$ be the vorticity associated to the stationary Euler flow of Theorem \ref{thm:main_Kolmo}, thus satisfying $\norm{\Omega_\eps-\Omega_{bar}(t=0)}_{L^2}=O(\eps)$. Then $\PP_\D\Omega_\eps$ is not dissipated at an enhanced rate: i.e.\ the solution $\Omega^\nu$ of the initial value problem
 \begin{equation}
  \begin{cases}
   \de_t\Omega^\nu+U^\nu\cdot\nabla\Omega^\nu=\nu\Delta\Omega^\nu,\\
   \Omega^\nu(0)=\Omega_\eps,
  \end{cases}
 \end{equation}
 on $\T^2$ satisfies for all $t\in[\frac{1}{2\nu},\frac{1}{\nu}]$ the lower bound
 \begin{equation}\label{eq:noL2decay}
  \norm{\PP_\D\Omega^\nu(t)}_{L^2}\gtrsim \norm{\PP_\D\Omega_\eps}_{L^2}.
 \end{equation}
\end{theorem}

\begin{proof}[About the proof of Theorem \ref{thm:main_Bar} (full details in Section \ref{sec:Bar})]
 To prove this we use the stationary states of the Euler equations constructed in Theorem \ref{thm:main_Kolmo}. We combine this here with the fact that the Navier-Stokes evolution preserves shears and uni-modal flows. The result is that one can still move away from the Kolmogorov flow in an almost stationary fashion, even in the Navier-Stokes equation. 
\end{proof}

We remark here once more on the crucial role played by the global degeneracy of $U_K$ on the square torus. This can be broken by considering $U_K$ on a \emph{rectangular} torus, as has been done in \cite{WZZ20}. In that setting, the kernel of the linearization $\Lin_K$ trivializes to include only shear flows again, and enhanced dissipation can be shown to hold not only linearly, but also below a threshold in the nonlinear problem.

\subsection{Perspectives}
While our results provide a striking look at the rich dynamics of solutions to $2d$ Euler and Navier-Stokes equations even near relatively simple, stationary flow configurations, they also open the door to many further questions. We briefly discuss here two areas that seem of particular relevance to us.

\paragraph{Local structure of $2d$ stationary Euler flows.}
Given a stationary solution of the $2d$ Euler equations, a natural and difficult question is whether one can describe all $2d$ Euler stationary states near it.  For some shear flows, it is possible to show that all smooth stationary states nearby are shear flows (this was done for the Couette flow in \cite{LZ11}, while our Theorems \ref{thm:main_rigid_intro} resp.\ \ref{thm:main_Pois} demonstrate it for the Kolmogorov flow on rectangular tori resp.\ the Poiseuille flow). 
In \cite{CS12}, set on general domains homeomorphic to an annulus, the authors establish a one-to-one correspondence between stationary states near a base "non-degenerate" state and their distribution function (similar to the case of the Couette flow); in the recent work \cite{CDG20} 
certain Liouville-type theorems are established (in the spirit of \cite{HN17,HN19}), which show that suitable steady solutions with no stagnation points occupying a two-dimensional periodic channel must have certain structural symmetries.

This is manifestly false for the Kolmogorov flow on $\mathbb{T}^2$ since any neighborhood of the Kolmogorov flow contains a four-dimensional set of solutions to $\Delta\psi=-\psi$. A natural question is whether these are the only "extra" solutions near the Kolmogorov flow. Our construction in Theorem \ref{thm:main_Kolmo} shows that there are other solutions and that the local structure of the set of $2d$ Euler stationary states near Kolmogorov is much richer. However, a  characterization of the full set of stationary solutions near the Kolmogorov flow on $\mathbb{T}^2$ is an outstanding open problem. In Theorem \ref{thm:main_Kolmo}, we find one non-trivial ``branch'' of solutions leaving Kolmogorov in a certain direction, but we also show that there cannot be any ``branches'' in certain other directions (see Proposition \ref{prop:someflexT2}). It seems highly non-trivial to characterize all these branches since there is balance between freedom and rigidity. 

\paragraph{Bar states and dipoles in $2d$ Navier-Stokes.}
Besides the bar state $\Omega_{bar}(t,y)=-\e^{-\nu t}\cos(y)$, the Navier-Stokes equations on $\T_\delta^2$ admit another explicit solution given by 
\begin{align}
\Omega_{dip}(t,y)=-\e^{-\nu t}\cos(y)-\e^{-\frac{\nu}{\delta^2} t}\cos (x/\delta),\qquad \delta\in(0,1],
\end{align}
known as \emph{dipole state}. Even at the linearized level, the questions of stability and enhanced dissipation properties of $\Omega_{dip}$ remain unsolved, in both the square and rectangular torus cases.
An interesting analysis in this direction has been carried out in \cite{BCS19}, following the work \cite{BW13}. In particular, evidence was provided there to 
show that $\Omega_{dip}$ is a (local) attractor in the square torus case $\delta=1$, while for $\delta<1$, $\Omega_{bar}$ is the asymptotic end state configuration, at least for small perturbations. While the latter statement on the nonlinear stability of $\Omega_{bar}$ was proven rigorously in \cite{WZZ20},
the case of the square torus is completely open. Our result, however, points strongly in the direction of confirming the predictions of  \cite{BCS19}. In  particular, Theorem \ref{thm:main_Bar} shows that $\Omega_{bar}$ is \emph{not} a local attractor for nearby perturbations on $\T^2$.

\subsection{Plan of the Article}
Section \ref{sec:Kolmo} lies at the heart of this article, and begins by establishing Theorem \ref{thm:main_Kolmo}. First we construct nontrivial stationary states near $U_K$ on $\mathbb{T}^2$ using a two-step contraction mapping argument in $H^2$. Second, we show that these stationary states can be taken to be arbitrarily close to $U_K$ in the analytic norm in Section \ref{ssec:elliptic}. In Section \ref{ssec:Kolmo_obstruct}, we demonstrate Proposition \ref{prop:main_someflex}, showing that not every linearly neutral direction gives rise to a nonlinear steady state near $U_K$. Section \ref{ssec:Kolmo_rigid} then gives the proof of rigidity on rectangular tori (Theorem \ref{thm:main_rigid_intro}): all stationary states near $U_K$ on $\mathbb{T}^2_\delta$ are shears when $\delta>0$ is not an integer. 

Section \ref{sec:Bar} is devoted to the Navier-Stokes equations, showing that no nonlinear enhanced dissipation result can hold near the bar state $U_{bar}$ on $\T^2$ (Theorem \ref{thm:main_Bar}).

Finally, we prove the rigidity result Theorem \ref{thm:main_Pois} for traveling waves near the Poiseuille flow $U_P$ in Section \ref{sec:Pois}.

\section{Stationary States near Kolmogorov flow}\label{sec:Kolmo}
In this section we investigate the existence of stationary states near the Kolmogorov flow $U_K=(\sin(y),0)$ on square or rectangular tori. To begin, we note that any nearby shear is trivially a stationary solution as well. In the specific setting of the \emph{square torus} $\T^2$, one verifies directly that in addition, flows of the form $\cos(y)+a\cos(x)+b\sin(x)$ are stationary, provided $a,b\in\R$ small enough. This already hints at the global degeneracy of this particular problem.

To understand the difficulties involved in finding a larger class of non-trivial stationary states near the Kolmogorov flow on $\T^2$, let us try to (formally) search for a solution of the $2d$ Euler equations of the form
\begin{equation}
\Omega_\eps = -\cos(y)+\sum_{j=1}^\infty \eps^j \omega_j(x,y),
\end{equation} 
with $\eps$ a small parameter and vorticity $\Omega_\eps$ non-shear and not just a solution of $\Delta\Omega_\eps=-\Omega_\eps$. By stationarity, the perturbation $\omega_\eps:=\sum_{j=1}^\infty\eps^j\omega_j$ has to satisfy the linearized equation
\begin{equation}
 \Lin_K\omega_\eps=-u_\eps\cdot\nabla\omega_\eps,\quad u_\eps=\nabla^\perp\Delta^{-1}\omega_\eps,\quad \Lin_K=\sin(y)(1+\Delta^{-1})\de_x,
\end{equation}
or equivalently
\begin{equation}\label{iterative}
\Lin_K\omega_k=-\sum_{j=1}^{k-1} u_j\cdot\nabla \omega_{k-j}, \qquad u_j=\nabla^\perp \Delta^{-1}\omega_j.
\end{equation} 
We can therefore hope to solve for $\omega_k$ given $\omega_1,...,\omega_{k-1}$. Of course, this method is unlikely to work directly since there is a clear loss of derivatives in this process. However, there are even more fundamental problems: the global degeneracy of $U_K$ on $\T^2$ is witnessed by the fact that the operator $\Lin_K$ is not invertible (we have $\ker\Lin_K=\{\cos(x),\sin(x)\}\cup\{f\in L^2:\de_x f\equiv 0\}$), and the solvability conditions for an equation of the form $\mathcal{L}_Kf=g$ are complicated. In particular, we would need to know that, at each step, the function
\begin{equation}
\frac{1}{\sin(y)}\sum_{j=1}^{k-1} u_j\cdot\nabla \omega_{k-j}
\end{equation}  
is \emph{smooth}, mean-free in $x$ only, and orthogonal to $\sin(x)$ and $\cos(x)$.

When $k=1$,  we see that $\omega_1=G(y)+a \sin(x)+b\cos(x)$, 
and we are free to choose $a,b\in\mathbb{R}$ and $G\in C^1(\mathbb{T})$ is mean-free. On the one hand, any non-trivial choice of $G,a,b$
will produce, through the nonlinearity, non-shear modes in $\omega_2$. On the other hand, the solvability of \eqref{iterative}
needs to be preserved, reducing drastically the degrees of freedom. Although it is not clear a priori whether this formal process can even be continued for all $k$, using the freedom of choice of $\omega_k$ at each step one can show the existence of non-shear formal power series solutions. The loss of derivatives in this process, however, makes it very difficult to rigorously show that the series converges even to an $L^2$ solution.

To get around the derivative loss, we choose to construct stationary solutions through the semilinear equation \eqref{eq:stat_sol} instead, branching
away from the respective equation \eqref{eq:psiKolmo} that the Kolmogorov flow satisfies. 

This culminates in Theorem \ref{thm:main_Kolmo}, which is established in Sections \ref{ssec:KolmoH2}-\ref{ssec:elliptic}. We proceed as follows: First we prove the corresponding statement for stream functions in $H^2(\T^2)$ in Proposition \ref{prop:KolmoH2} below. Lemma \ref{lem:psi_eps_exp} in Section \ref{ssec:Kolmo_nontriv} then demonstrates that these stream functions are indeed non-trivial in the sense that they do not lie in the kernel $\ker\Lin_K$. Via an elliptic regularity type argument our stationary states can subsequently be upgraded to have analytic regularity -- see Proposition \ref{prop:elliptic} as well as Lemma \ref{lem:analytic} and Corollary \ref{cor:analytic} in Section \ref{ssec:elliptic}. 

Following this, we establish some obstructions to a natural generalization of Proposition \ref{prop:someflexT2} of Section \ref{ssec:Kolmo_obstruct}, as well as a rigidity Theorem \ref{thm:main_rigid_intro} for rectangular tori in Section \ref{ssec:Kolmo_rigid}.
Now let us state the results that combine to give Theorem \ref{thm:main_Kolmo}.

\begin{proposition}\label{prop:KolmoH2}
 There exists $\eps_0>0$ such that for any $0<\eps\leq\eps_0$ there exist functions $\Psi_\eps\in H^2(\T^2)$ and $F_\eps:\R\to\R$ satisfying
 \begin{equation}\label{eq:full_stat_sol}
  \Delta\Psi_\eps=F_\eps(\Psi_\eps)
 \end{equation}
 and
 \begin{equation}\label{eq:H2_diff}
  \norm{\cos(y)-\Psi_\eps}_{H^2(\T^2)}= O(\eps),
 \end{equation}
 with
 \begin{equation}\label{eq:catseye}
  \ip{\Psi_\eps,\cos(x)\cos(4y)}=-\eps^2\frac{\pi^2}{128}+O(\eps^3).
 \end{equation}
 Here the functions $F_\eps$ can be chosen to be polynomials of degree five.
\end{proposition}

\begin{remark}\label{rem:KolmoH2}
 We comment on a few extra details.
 \begin{enumerate}
  \item More precisely, $\Psi_\eps$ can be computed to have the expansion
 \begin{equation}
 \begin{aligned}
  \Psi_\eps&=\cos(y)+\eps\left[\cos(x)+c_0\cos(3y)-c_1\cos(5y)\right]\\
  &\quad +\eps^2\left[-c_2\cos(x)\cos(4y)-\frac{1}{32}b_1\cos(3y)-c_3\cos(7y)+c_4\cos(9y)\right]\\
  &\quad +O(\eps^3),
 \end{aligned} 
 \end{equation}
 as is shown in Lemma \ref{lem:psi_eps_exp} (where also $c_0,\ldots,c_4, b_1$ are given).
 \item Furthermore, one can easily modify our arguments to show that many such families $(\Psi_\eps)_\eps$ exist. Indeed, one way to see this is simply to modify our construction of the functions $F_\eps$ by adding polynomials with coefficients of order $\eps^2$.
 \end{enumerate}
\end{remark}

In order to give the precise analyticity statement, for $\lambda>0$ we introduce the Gevrey-$1$ space $\mathcal{G}^\lambda(\T^2)$ as the Banach space of $L^2$ functions, whose norm
\begin{equation}
 \norm{f}_{\mathcal{G}^\lambda}:=\norm{\e^{\lambda\abs{D}}f}_{L^2}=\norm{\e^{\lambda\abs{k}}\hat{f}(k)}_{\ell^2(k)}<+\infty
\end{equation}
is finite. Clearly, such functions are analytic, with radius of analyticity $\lambda$.

\begin{proposition}\label{prop:elliptic}
 The functions $\Psi_\eps$ constructed in Proposition \ref{prop:KolmoH2} are in fact analytic, i.e.\ $\Psi_\eps\in C^\omega(\T^2)$, and there exists $\lambda>0$ and a constant $M>0$, both independent of $\eps>0$, such that they satisfy
  \begin{equation}\label{eq:analytic_prox}
  \norm{\cos(y)-\Psi_\eps}_{\mathcal{G}^\lambda(\T^2)}\leq M\cdot \eps.
 \end{equation}
\end{proposition}

\subsection{Proof of Proposition \ref{prop:KolmoH2}}\label{ssec:KolmoH2}
We give the proof of Proposition \ref{prop:KolmoH2} in the following subsections: First we discuss the setup of the basic construction in Section \ref{ssec:Kolmo_setup}, which then leads to a contraction argument (Section \ref{ssec:Kolmo_contr}), proving \eqref{eq:full_stat_sol}. After this we can work with the explicit expansion of our functions to establish \eqref{eq:H2_diff}, i.e.\ the presence of modes, which guarantees that the associated flows do not lie in the kernel of the linearization $\Lin_K$. %Finally, Proposition \ref{prop:elliptic} is proved in Section \ref{ssec:elliptic}.

\subsubsection{Setup}\label{ssec:Kolmo_setup}
Our goal is to find stream functions $\Psi_\eps:\T^2\to\R$, for $\eps>0$ sufficiently small, that satisfy
\begin{equation}\label{eq:stat-stream}
 \Delta \Psi_\eps=F_\eps(\Psi_\eps),
\end{equation}
and are ``close'' to the Kolmogorov flow $\Psi_K:\T^2\to\R$, $(x,y)\mapsto \cos(y)$. Since this flow itself satisfies \eqref{eq:stat-stream} with $F_K:\R\to\R$, $z\mapsto -z$, we make the ansatz
\begin{equation}
 \Psi_\eps=\Psi_K+\eps\psi,\quad F_\eps=F_K+\eps f,
\end{equation}
for perturbations $\psi:\T^2\to\R$ and $f:\R\to\R$ which are to be determined. Plugging this into \eqref{eq:stat-stream}, we obtain that $(f,\psi)$ need to satisfy $\Delta\psi+\psi=f(\cos(y)+\eps\psi)$.
Since $\cos(x)\in\ker(\Delta+1)$, we may replace $\psi$ by $\cos(x)+\psi$, which gives us the following equation to be solved:
\begin{equation}\label{eq:psi_elliptic}
 \Delta\psi+\psi=f(\cos(y)+\eps\cos(x)+\eps\psi),\qquad\text{ with } \psi\perp\ker(\Delta+1).
\end{equation}
Taking $f$ as a quintic polynomial (with coefficients $A,B\in\R$ to be determined as functionals of $\psi$ and $\eps>0$)
\begin{equation}
f(A,B;s)=As+Bs^3+\frac{1}{5}s^5,
\end{equation}
we obtain
\begin{equation}\label{eq:base}
\begin{aligned}
 \Delta\psi+\psi &= A\cos(y)+B\cos^3(y)+\frac{1}{5}\cos^5(y)\\
 &\quad +\eps\psi\Big(A+3B\cos^2(y)+\cos^4(y)\Big)\\
 &\quad +\eps\cos(x)\Big(A+3B\cos^2(y)+\cos^4(y)\Big)\\
 &\quad +R(B,\psi,\eps;x,y),
\end{aligned}
\end{equation}
with 
\begin{equation}\label{eq:rest}
\begin{aligned}
 R(B,\psi,\eps;x,y)&=\eps^2(\psi+\cos(x))^2\left(3B\cos(y)+2\cos^3(y)\right)\\
 &\quad +\eps^3(\psi+\cos(x))^3\left(B+2\cos^2(y)\right)\\
 &\quad +\eps^4(\psi+\cos(x))^4 \cos(y)+\eps^5(\psi+\cos(x))^5.
\end{aligned} 
\end{equation}
For simplicity we assume further that $\psi$ is an even function in both $x$ and $y$ (separately). Therefore,
a compatibility condition in order for \eqref{eq:base} to have a solution is that
\begin{equation}
 \ip{f(A,B;\cos(y)+\eps\cos(x)+\eps\psi),\cos(x)}=\ip{f(A,B;\cos(y)+\eps\cos(x)+\eps\psi),\cos(y)}=0.
\end{equation}
These equations can be viewed as restrictions for the two coefficients $A=A(\psi;\eps)$ and  $B=B(\psi;\eps)$: Plugging in \eqref{eq:base} and using that $\int_\T\cos^4(y)\dd y=\frac{3\pi}{4}$ and $\int_\T\cos^6(y)\dd y=\frac{5\pi}{8}$, we arrive at the two conditions
\begin{align}
 &(A(\psi;\eps),B(\psi;\eps))\cdot V_1=-\frac18-\eps\frac{3 B(\psi;\eps)}{2\pi^2}\ip{\psi,\cos^3(y)}-\eps \frac{1}{2\pi^2}\ip{\psi,\cos^5(y)}-\frac{1}{2\pi^2}\ip{R,\cos(y)}, \label{eq:first}\\
 &(A(\psi;\eps),B(\psi;\eps))\cdot V_2=-\frac{3}{8}\left[1+\frac{4}{3\pi^2}\ip{\psi,\cos^4(y)\cos(x)}\right] -\frac{1}{2\pi^2\eps}\ip{R,\cos(x)}\label{eq:second},
\end{align}
where
\begin{equation}
 V_1:=(1,\frac{3}{4}),\quad V_2:=(1,\frac{3}{2}\Big[1+\frac{1}{\pi^2}\ip{\psi,\cos^2(y)\cos(x)}\Big]).
\end{equation}
Observe that if $\abs{\ip{\psi,\cos^2(y)\cos(x)}}$ is sufficiently small, the vectors $V_1$ and $V_2$ are not parallel. Together with the prior remarks, this motivates our definition of the function space $X$ we will work in as
\begin{equation}
\begin{aligned}
 X:=\Big\{&\psi\in H^2: \psi(-x,y)=\psi(x,-y)=\psi(x,y),\quad \psi\perp \cos(y),\cos(x),\\
 &\abs{\ip{\psi,\cos^2(y)\cos(x)}}+\abs{\ip{\psi,\cos^4(y)\cos(x)}}\leq \frac{1}{100}, \quad \norm{\psi}_{H^2}\leq 10\Big\}.
\end{aligned} 
\end{equation} 

\begin{lemma}\label{lem:AB}
 There exists $\eps_1>0$ such that if $\psi\in X$, for $0\leq\eps\leq\eps_1$ the above relations \eqref{eq:first}, \eqref{eq:second} inductively define $(a_j(\psi))_{j\geq 0},(b_j(\psi))_{j\geq 0}\subset\R$ with the property that
 \begin{equation}\label{eq:AB_exp}
  A(\psi;\eps):=\sum_{j\geq 0}a_j(\psi)\eps^j,\qquad B(\psi;\eps):=\sum_{j\geq 0}b_j(\psi)\eps^j
 \end{equation}
 are well-defined, uniformly bounded for $\psi\in X$, and satisfy both \eqref{eq:first} and \eqref{eq:second}.
 
 Moreover, the maps
 \begin{equation}
  \psi\mapsto a_j(\psi), \qquad \psi\mapsto b_j(\psi),\qquad j\geq0,
 \end{equation}
 are Lipschitz continuous on $L^2$ with constants $L_j\leq L^j$ for some $L>0$, and the maps
 \begin{equation}
  \psi\mapsto a_0(\psi), \qquad \psi\mapsto b_0(\psi),
 \end{equation}
 are Lipschitz continuous on $\dot{H}^2$ (and thus also $H^2$) with constant $\tilde{L}_0\leq\frac{1}{4\pi}$.
\end{lemma}

\begin{proof}
 Subtracting \eqref{eq:first} from \eqref{eq:second} we obtain the following closed form for $B(\psi;\eps)$:
 \begin{equation}\label{eq:Bitself}
 \begin{aligned}
  \left[\frac{3}{4}+\frac{3}{2\pi^2}\ip{\psi,\cos^2(y)\cos(x)}\right]B(\psi;\eps)&=-\left[\frac{1}{4}+\frac{1}{2\pi^2}\ip{\psi,\cos^4(y)\cos(x)}\right]\\
  &\quad +\frac{1}{2\pi^2}\left[-\frac{1}{\eps}\ip{R,\cos(x)}+3\eps B(\psi;\eps)+\eps\ip{\psi,\cos^5(y)}\right]\\
  &\quad +\frac{1}{2\pi^2}\ip{R,\cos(y)}.
 \end{aligned}
 \end{equation}
 Inserting here the expansion \eqref{eq:AB_exp} for $B(\psi;\eps)$ and comparing coefficients (in $\eps$) shows that $b_j(\psi)$ can be inductively defined from $\{b_{k}(\psi)\}_{j-3\leq k\leq j-1}$. In fact, the map $\{b_{k}(\psi)\}_{j-3\leq k\leq j-1}\mapsto b_j(\psi)$ is a linear map with coefficients that are uniformly bounded in $\psi\in X$. Hence for $M$ sufficiently large we have
 \begin{equation}
  \abs{b_j(\psi)}\leq M^j,
 \end{equation}
 and the series expansion for $B(\psi;\eps)$ converges for $0\leq\eps< M^{-1}$. The same holds for $A(\psi,\eps)$, since we can now simply use \eqref{eq:first} to find its expansion, e.g.\ we have
 \begin{equation}\label{eq:afromb}
  a_0(\psi)=-\frac{3}{4}b_0(\psi)-\frac{1}{8},
 \end{equation}
 and similarly for $a_j(\psi)$, $j\geq 1$, which can be inductively defined from $B(\psi;\eps)$ and $\{a_k(\psi)\}_{j-3\leq k\leq j-1}$.
 
 It remains to prove the claimed Lipschitz property. For $j\geq 1$ this follows directly from the recursive construction of the coefficients. Regarding $j=0$, we observe that by \eqref{eq:Bitself} we have 
 \begin{equation}\label{eq:b0}
  b_0(\psi)=B(\psi;0)=-\frac{1+\frac{2}{\pi^2}\ip{\psi,\cos^4(y)\cos(x)}}{3+\frac{6}{\pi^2}\ip{\psi,\cos^2(y)\cos(x)}}=:\frac{n(\psi)}{d(\psi)}.
  \end{equation}
 Towards finding the $\dot{H}^2$ Lipschitz constants, we note that since $\psi\in X$, there holds
 \begin{equation}
 \begin{aligned}
  \ip{\psi,\cos^4(y)\cos(x)}&=\ip{\psi,\left(\frac{1}{8}\cos(4y)+\frac{1}{2}\cos(2y)\right)\cos(x)},\\
  \ip{\psi,\cos^2(y)\cos(x)}&=\ip{\psi,\frac{1}{2}\cos(2y)\cos(x)}.
\end{aligned}
 \end{equation}
 Consequently we obtain
 \begin{equation}
  \abs{n(\psi_1)-n(\psi_2)}\leq\norm{\psi_1-\psi_2}_{\dot{H}^2}\norm{\left(\frac{1}{8}\cos(4y)+\frac{1}{2}\cos(2y)\right)\cos(x)}_{\dot{H}^{-2}}
 \end{equation}
 and
 \begin{equation}
  \abs{d(\psi_1)-d(\psi_2)}\leq\norm{\psi_1-\psi_2}_{\dot{H}^2}\norm{\frac{1}{2}\cos(2y)\cos(x)}_{\dot{H}^{-2}}.
 \end{equation}
 Now we compute
 \begin{equation}
  \norm{\left(\frac{1}{8}\cos(4y)+\frac{1}{2}\cos(2y)\right)\cos(x)}_{\dot{H}^{-2}}=\pi\left[17^{-2}\frac{1}{8^2}+5^{-2}\frac{1}{2^2}\right]^{1/2}\leq\frac{\pi}{9}
 \end{equation}
 and
 \begin{equation}
  \norm{\frac{1}{2}\cos(2y)\cos(x)}_{\dot{H}^{-2}}=\frac{\pi}{10}.
 \end{equation}
 The $\dot{H}^2$ (and thus also $H^2$) Lipschitz constant of $b_0$ is thus bounded by
 \begin{equation}\label{eq:kobe}
  \frac{1}{3-\frac{6}{100\pi^2}}\frac{2}{\pi^2}\frac{\pi}{9}+\left(\frac{1}{3-\frac{6}{100\pi^2}}\right)^2\left(1+\frac{1}{50\pi^2}\right)\frac{6}{\pi^2}\frac{\pi}{10}\leq\frac{1}{4\pi}.
 \end{equation}
 In view of \eqref{eq:afromb} this bound also holds for the Lipschitz constant of $a_0$.
\end{proof}

We conclude this section with some direct properties of the function $f$ thus constructed.
\begin{lemma}\label{lem:f_prop}
 Let $\psi,\psi_j\in X$, $j\in\{1,2\}$, and construct $A(\psi;\eps)$, $B(\psi;\eps)$ as in Lemma \ref{lem:AB}. Then we have for $\eps>0$ sufficiently small that
 \begin{align}
 &\abs{A(\psi;\eps)},\abs{B(\psi;\eps)}\leq 1,\label{eq:ABbound}\\
 &\abs{A(\psi_1;\eps)-A(\psi_2;\eps)},\abs{B(\psi_1;\eps)-B(\psi_2;\eps)}\leq \frac{1}{3\pi}\norm{\psi_1-\psi_2}_{\dot{H}^2},\label{eq:ABLipH2}
 \end{align}
 and
 \begin{align}
 &\norm{R(B(\psi;\eps),\psi,\eps;\cdot,\cdot)}_{L^2}\lesssim \eps^2, \label{eq:restbd}\\
 &\norm{R(B(\psi_1;\eps),\psi_1,\eps;\cdot,\cdot)-R(B(\psi_2;\eps),\psi_2,\eps;\cdot,\cdot)}_{L^2}\lesssim \eps^2\norm{\psi_1-\psi_2}_{L^2}.\label{eq:restLip}
 \end{align}
\end{lemma}

\begin{proof}
 The bounds \eqref{eq:ABbound} follow directly from the power series construction of $A$ and $B$. The estimate
 \eqref{eq:ABLipH2} follows from the expansion \eqref{eq:AB_exp} and the bounds for the Lipschitz constants of the zero order terms $a_0(\psi),b_0(\psi)$ in Lemma \ref{lem:AB}.
 Similarly, \eqref{eq:restbd} and \eqref{eq:restLip} follow directly from construction.
\end{proof}

\subsubsection{Contraction Argument}\label{ssec:Kolmo_contr}
We will now construct solutions of \eqref{eq:base} as contractions in $X$. To this end, define
\begin{equation}
 K_\eps:X\to H^2,\quad \psi\mapsto \left[(x,y)\mapsto (1+\Delta)^{-1}f(A(\psi;\eps),B(\psi,\eps);\cos(y)+\eps\cos(x)+\eps\psi)\right],
\end{equation}
where $A(\psi;\eps)$ and $B(\psi;\eps)$ are constructed as in Lemma \ref{lem:AB}

\begin{proposition}\label{prop:contract}
For $\eps>0$ small enough, $K_\eps$ defines a contraction on $(X,\norm{\cdot}_{H^2})$.
\end{proposition}

\begin{proof}
Let $0<\eps<\eps_1$, so that by Lemma \ref{lem:AB} the coefficients $A(\psi;\eps)$ and $B(\psi;\eps)$ are well-defined. 

First we show that $K_\eps$ maps $X$ into itself. By construction it is clear that if $\psi$ is even in $x$ and $y$ (separately), then so is $K_\eps(\psi)$. Due to the smoothing property of $(\Delta +1)$ and the fact that $H^2$ forms an algebra, it suffices to prove $H^2\to L^2$ bounds on
\begin{equation}
 \widetilde{K_\eps}:\psi\mapsto f(A(\psi;\eps),B(\psi,\eps);\cos(y)+\eps\cos(x)+\eps\psi).
\end{equation}
Since
\begin{align}
 \abs{f(A(\psi;\eps),B(\psi;\eps); \cos(y)+\eps \cos(x)+\eps\psi)}\leq 1+C\eps\abs{\psi}^5
\end{align}
we find
\begin{align} 
 \norm{f(A(\psi;\eps),B(\psi;\eps); \cos(y)+\eps \cos(x)+\eps\psi)}_{L^2}\leq 10,
\end{align}
so that
\begin{align} 
\norm{K_\eps(\psi)}_{H^2}\leq \norm{\widetilde{K_\eps}(\psi)}_{L^2} \leq 10.
\end{align}
Moreover, one computes directly that 
\begin{equation}
 \abs{\ip{K_\eps(\psi),\cos^2(y)\cos(x)}}+\abs{\ip{K_\eps(\psi),\cos^4(y)\cos(x)}}\lesssim \eps,
\end{equation}
and thus for $\eps$ small enough $K_\eps(X)\subset X$.

To show that $K_\eps$ defines a contraction, let $\psi_j\in X$, $j\in\{1,2\}$, and define
\begin{align}
 G_j=\cos(y)+\eps\cos(x)+\eps\psi_j, \quad \text{ with }\norm{G_j}_{L^2}\leq 2.
\end{align}
Then
\begin{equation}
\begin{aligned}
 \widetilde{K_\eps}(\psi_1)-\widetilde{K_\eps}(\psi_2)&=f(A(\psi_1;\eps),B_\eps(\psi_1);G_1))-f(A(\psi_2;\eps),B(\psi_2;\eps);G_2))\\
 &=(A(\psi_1;\eps)-A(\psi_2;\eps))G_1+ \eps A(\psi_2;\eps) (\psi_1-\psi_2) \\
 &\qquad +(B(\psi_1;\eps)-B(\psi_2;\eps)) G_1^3+ \eps B(\psi_2;\eps) (\psi_1-\psi_2)\left[G_1^2 +G_1G_2+G_2^2\right]\\
 &\qquad+ \frac{\eps}{5} (\psi_1-\psi_2)\left[G_1^4 + G_1^3 G_2 + G_1^2 G_2^2 + G_1 G_2^3 + G_2^4\right].
\end{aligned} 
\end{equation}
Up to terms of order $\eps$ we then have to bound
\begin{equation}
\begin{aligned}
 &\norm{(a_0(\psi_1)-a_0(\psi_2))\cos(y)+(b_0(\psi_1)-b_0(\psi_2))\cos^3(y)}_{L^2}\\
 &\qquad\leq\frac{1}{4\pi}\left[\norm{\cos(y)}_{L^2}+\norm{\cos^3(y)}_{L^2}\right]\norm{\psi_1-\psi_2}_{\dot{H}^2}\\
 &\qquad =\frac{\sqrt{2}}{4}\left[1+\sqrt{\frac{5}{8}}\right]\norm{\psi_1-\psi_2}_{\dot{H}^2}\\
 &\qquad <\frac{2}{3}\norm{\psi_1-\psi_2}_{\dot{H}^2}\leq \frac{2}{3}\norm{\psi_1-\psi_2}_{H^2},
\end{aligned} 
\end{equation}
thanks to Lemma \ref{lem:AB}. This shows that 
\begin{equation}
 \norm{K_\eps(\psi_1)-K_\eps(\psi_2)}_{H^2}\leq \norm{\widetilde{K_\eps}(\psi_1)-\widetilde{K_\eps}(\psi_2)}_{L^2}\leq \left(\frac{2}{3}+O(\eps)\right)\norm{\psi_1-\psi_2}_{H^2},
\end{equation}
and for $\eps>0$ sufficiently small we thus obtain a contraction.
\end{proof}

\subsection{Non-triviality of the Stationary Modes}\label{ssec:Kolmo_nontriv}
Now for given $\eps>0$, let $\psi_\eps\in X$ be the fixed point of  $K_\eps$ in $X$, well-defined by virtue of Proposition \ref{prop:contract}. We conclude the proof of Proposition \ref{prop:KolmoH2} by demonstrating that $\psi_\eps$ has nontrivial $x$ modes, in the sense that the associated flows are not in the kernel of the linearization $\Lin_K$.
\begin{lemma}\label{lem:psi_eps_exp}
 For sufficiently small $\eps>0$ as in Proposition \ref{prop:contract}, let $\psi_\eps$ be the fixed point of $K_\eps$. Then $\psi_\eps$ has the expansion
 \begin{equation}\label{eq:psi_eps_exp}
 \begin{aligned}
  \psi_\eps&=c_0\cos(3y)-c_1\cos(5y)\\
  &\quad +\eps\left[-c_2\cos(x)\cos(4y)-\frac{1}{32}b_1\cos(3y)-c_3\cos(7y)+c_4\cos(9y)\right]\\
  &\quad +O(\eps^2),
 \end{aligned} 
 \end{equation}
with coefficients that can be explicitly computed as $(c_0,c_1,c_2,c_3,c_4)=(\frac{1}{384},\frac{1}{1920},\frac{1}{128},-\frac{c_0}{768},\frac{c_1}{1280})$ and $b_1=-\frac{7}{7680}$. 
\end{lemma}
In particular, from \eqref{eq:psi_eps_exp} we directly have
\begin{equation}
 \ip{\psi_\eps,\cos(kx)\cos(ly)}=\begin{cases}-\eps\frac{\pi^2}{128}+O(\eps^2), &(k,l)=(1,4),\\ O(\eps^2), &\text{else, with }k\neq 0, l\neq 3,5,\end{cases}
\end{equation}
from which \eqref{eq:catseye} follows: it suffices to recall that $\Psi_\eps=\cos(y)+\eps\cos(x)+\eps\psi_\eps$.

\begin{proof}[Proof of Lemma \ref{lem:psi_eps_exp}]
 We recall that for a given $\eps>0$, by construction the function $\psi_\eps$ satisfies the identity \eqref{eq:base}, which we expand and restate here for convenience:
 \begin{equation}\label{eq:base_exp}
 \begin{aligned}
  \Delta\psi_\eps+\psi_\eps&=\cos(y)\left[A+\frac{3}{4}B+\frac{1}{8}\right]+\cos(3y)\left[\frac{1}{4}B+\frac{1}{16}\right]+\cos(5y)\left[\frac{1}{80}\right]\\
  &\quad +\eps(\psi_\eps+\cos(x))\left[(A+\frac{3}{2}B+\frac{3}{8})+\cos(2y)(\frac{3}{2}B+\frac{1}{2})+\cos(4y)\frac{1}{8}\right]\\
  &\quad +O(\eps^2).
 \end{aligned} 
 \end{equation}
 Here the coefficients $A=A(\psi_\eps;\eps),B=B(\psi_\eps;\eps)$ are \emph{fixed} and given explicitly by solving the system \eqref{eq:first}, \eqref{eq:second}.
 
 Via the relation 
 \begin{equation}\label{eq:inv_numbers}
  \ip{(1+\Delta)\psi_\eps,\cos(kx)\cos(ly)}=(1-k^2-l^2)\ip{\psi_\eps,\cos(kx)\cos(ly)},
 \end{equation}
 we may thus successively determine an expansion of $\psi_\eps$ by testing \eqref{eq:base_exp} with $\cos(kx)\cos(ly)$.
 
 Since by \eqref{eq:afromb} and \eqref{eq:b0} we have $a_0(\psi_\eps)=\frac{1}{8}+O(\eps)$ and $b_0(\psi_\eps)=-\frac{1}{3}+O(\eps)$ with $a_0(\psi_\eps)+\frac{3}{4}b_0(\psi_\eps)+\frac{1}{8}=0$, this yields
 \begin{equation}\label{eq:base_exp2}
 \begin{aligned}
  \Delta\psi_\eps+\psi_\eps&=-\frac{1}{48}\cos(3y)+\frac{1}{80}\cos(5y)\\
  &\quad +\eps(\psi_\eps|_{\eps=0}+\cos(x))\frac{1}{8}\cos(4y)+\eps\cos(y)\left[a_1+\frac{3}{4}b_1\right]+\eps\cos(3y)\left[\frac{1}{4}b_1\right]\\
  &\quad +O(\eps^2),
 \end{aligned} 
 \end{equation}
 where $a_1,b_1$ can be computed explicitly through the equations \eqref{eq:first}, \eqref{eq:second}. Hence from \eqref{eq:inv_numbers} it follows that $\psi_\eps|_{\eps=0}=c_0\cos(3y)-c_1\cos(5y)$ with $c_0=\frac{1}{384}$ and $c_1=\frac{1}{1920}$. Reinserting this into the second line of \eqref{eq:base_exp2} and computing the terms of order $\eps$ then yields the claim \eqref{eq:psi_eps_exp}.
\end{proof}

\subsection{Proof of Proposition \ref{prop:elliptic} -- Elliptic Regularity}\label{ssec:elliptic}
In this section we prove Proposition \ref{prop:elliptic}. Since $\Psi_\eps-\cos(y)=\eps\cos(x)+\eps\psi_\eps$, it is enough to show that in fact our fixed point $\psi_\eps\in C^\omega(\T^2)$ is an analytic function with uniform in $\eps$ bounded Gevrey-$1$ norm. This is the content of the following corollary.

\begin{corollary}\label{cor:analytic}
 For $\eps>0$ sufficiently small as in Proposition \ref{prop:contract}, let $\psi_\eps\in H^2(\T^2)$ be the fixed point of $K_\eps$. Then we have that $\psi_\eps\in C^\omega(\T^2)$ and there exist constants $\lambda>0$ and $M>0$, both independent of $\eps>0$, such that
 \begin{equation}
  \norm{\psi_\eps}_{\mathcal{G}^{\lambda}}\leq M.
 \end{equation}
\end{corollary}

The key point here is that $\psi_\eps$ satisfies the equation \eqref{eq:psi_elliptic}, which is an elliptic, semilinear equation with analytic coefficients. Our Corollary \ref{cor:analytic} is then a direct consequence of the following slightly more general lemma regarding ``elliptic regularity''.
\begin{lemma}\label{lem:analytic}
Let $\mathscr{L}$ be a linear, constant coefficient partial differential operator, for which there exists a constant $C_{\mathscr{L}}>0$ such that 
for $f\in L^2$, $f\not\in\ker \mathscr{L}$, there holds
\begin{equation}\label{eq:L_gain}
 C_{\mathscr{L}}\norm{\mathscr{L} f}_{L^2}\geq\norm{f}_{\dot{H}^{2}},
\end{equation}
and let $a_k\in C^\omega(\T^2)$, $0\leq k\leq n$, be analytic functions, with $C_a>0$ such that for all $\ell\in\N_0$
\begin{equation}\label{eq:coeff_bd}
 \norm{a_k}_{H^\ell}\leq (C_a)^\ell\cdot \ell!,\quad 0\leq k\leq n.
\end{equation}
If $\varphi\in H^2(\T^2)$ solves the semilinear partial differential equation 
\begin{equation}\label{eq:proto_pde}
 \mathscr{L}\varphi=\sum_{k=0}^n a_k\varphi^k, 
\end{equation}
then $\varphi\in C^\omega(\T^2)$, and there exist constants $\lambda>0$ and $C_\ast>0$, depending only on $C_a$ and $C_\mathscr{L}$, such that
\begin{equation}
 \norm{\varphi}_{\mathcal{G}^\lambda}\leq 2C_\ast.
\end{equation}
\end{lemma}

\begin{remark}
\begin{enumerate}
 \item As the proof shows, it suffices to impose the requirement \eqref{eq:L_gain} on the solution of \eqref{eq:proto_pde}, rather than on a general $f\in L^2$. Moreover, one may allow $\mathscr{L}$ to have variable coefficients, provided suitable commutativity properties with derivatives hold.

 \item By tracking the constants in the proof of Lemma \ref{lem:analytic} one sees that the radius of analyticity $\lambda$ can be chosen to be of order $O(C_a^{-1})$.
 
 \item While the analytic regularity of solutions to general semilinear elliptic equations with analytic nonlinearity seems to be a classical result (see for example \cite[page 136]{BeJoSc_PDE}), we were not able to find a modern proof of this that also gives norm estimates for the solutions. We thus give the full result and proof.
\end{enumerate}

\end{remark}

We show next how this implies the claimed analyticity of the stationary solutions $\psi_\eps$ with uniform in $\eps$ Gevrey bounds.
\begin{proof}[Proof of Corollary \ref{cor:analytic}]
Fix $\eps>0$ as in the statement of Proposition \ref{prop:elliptic}. By construction \eqref{eq:psi_elliptic}, $\psi_\eps$ satisfies $\psi_\eps\not\in\ker(1+\Delta)$ and solves an equation of the form
\begin{equation}
 (\Delta+1)\psi_\eps=p(\cos(y)+\eps\cos(x)+\eps\psi_\eps).
\end{equation}
Here, $p:z\mapsto Az+Bz^3+\frac{1}{5}z^5$ is a real polynomial with \emph{fixed} coefficients $A=A(\psi_\eps;\eps),B=B(\psi_\eps;\eps)\in [-1,1]$ that are bounded uniformly in $\eps>0$ (for $\eps$ sufficiently small) -- see Lemma \ref{lem:f_prop}.
Expanding $p$ as a polynomial in $\psi_\eps$ we thus obtain uniformly bounded $c_{mn}\in\R$, $1\leq m,n\leq 5$, such that
\begin{equation}
 (\Delta+1)\,\psi_\eps=\sum_{k=0}^5 \Big(\sum_{m+n=5-k}c_{mn}\cos(y)^m\eps^n\cos(x)^n\Big)\cdot\psi_\eps^k.
\end{equation}
This is of the form \eqref{eq:proto_pde}, and the conditions of Corollary \eqref{cor:analytic} are satisfied, uniformly in $\eps$: We have $\widehat{(\Delta +1)\varphi}(k)=(-k^2+1)\hat\varphi(k)$, hence for $\varphi\not\in\ker (\Delta+1)$ there holds $\norm{(\Delta+1)\varphi}_{L^2}\geq C_1\norm{\varphi}_{\dot{H}^{2}}$ for some $C_1>0$, and there exists $C_2>0$ such that for any $\ell\in\N_0$ there holds
\begin{equation}
 \norm{\sum_{m+n=5-k}c_{mn}\cos(y)^m\eps^n\cos(x)^n}_{H^\ell}\leq C_2^\ell,\qquad 0\leq k\leq 5.
\end{equation}
Hence we may apply the result of Lemma \ref{lem:analytic} to obtain the claim.
\end{proof}

Finally, we conclude this section by giving the proof of Lemma \ref{lem:analytic}.
\begin{proof}[Proof of Lemma \ref{lem:analytic}]
 For simplicity of notation let us abbreviate the nonlinearity in \eqref{eq:proto_pde} as $\mathcal{N}(\varphi)$. Furthermore we will write $C_0>0$ for the constant in the algebra property
 \begin{equation}\label{eq:H2_alg}
  \norm{fg}_{H^2}\leq C_0\norm{f}_{H^2}\norm{g}_{H^2}
 \end{equation}
 of $H^2$.
 
 We show by induction that there exist constants $R>0$, $C_\ast>0$, such that for any multiindex $\alpha=(\alpha_1,\alpha_2)\in\N_0^2$ we have
 \begin{equation}\label{eq:ind_hyp}
  \Vert\partial^\alpha\varphi\Vert_{H^2}\leq C_\ast\frac{R^{\abs{\alpha}}\cdot \alpha!}{(\alpha_1+1)^2(\alpha_2+1)^2}.
 \end{equation}
For $\abs{\alpha}=0$ this is simply the statement that $\varphi\in H^2(\T^2)$. Note that by assumption \eqref{eq:coeff_bd} we may assume that the corresponding statement with constant $2C_a$ holds for $a_k$,
\begin{equation}\label{eq:ind_coeff}
 \Vert\partial^\alpha a_k\Vert_{H^2}\leq \frac{(2C_a)^{\abs{\alpha}}\cdot \alpha!}{(\alpha_1+1)^2(\alpha_2+1)^2},\qquad 0\leq k\leq n,
\end{equation}
and we thus assume in what follows that $R\geq 2C_\alpha$.

If $\abs{\alpha}\leq 2$, using the assumptions \eqref{eq:L_gain} and \eqref{eq:coeff_bd} as well as the algebra property \eqref{eq:H2_alg}, we can estimate
\begin{equation}
 \norm{\partial^\alpha\varphi}_{\dot{H}^2}\leq C_\mathscr{L} \norm{\partial^\alpha \mathscr{L} \varphi}_{L^2}\leq C_\mathscr{L} \norm{\mathcal{N}(\varphi)}_{H^2}\leq 4C_0^{n-1}C_\ast C_\mathscr{L} \cdot((C_a^2\cdot 2!)^{n}+\norm{\varphi}_{H^2}^n)\leq \frac{R^{\abs{\alpha}}\cdot\alpha!}{9},
\end{equation}
for $R>0$ chosen large enough.

Now assume that \eqref{eq:ind_hyp} holds for all $\abs{\beta}\leq\ell$, for some $\ell\geq 3$. Let $\beta=(\beta_1,\beta_2)\in\N_0^2$ with $\vert\beta\vert=\ell+1$. Then we have that for any $\gamma\leq\beta$ with $\abs{\gamma}=\vert\beta\vert-2=\ell-1$, there holds
\begin{equation}
 \Vert\partial^\beta\varphi\Vert_{\dot{H}^2}\leq C_\mathscr{L} \Vert\partial^\beta \mathscr{L} \varphi\Vert_{L^2}\leq C_\mathscr{L} \Vert\partial^\beta\mathcal{N}(\varphi)\Vert_{L^2}\leq C_\mathscr{L} \norm{\partial^{\gamma}\mathcal{N}(\varphi)}_{H^2}.
\end{equation}
Since by \eqref{eq:ind_hyp} and \eqref{eq:ind_coeff} both $a_k$ and $\varphi$ satisfy the same kind of bounds, to estimate $\partial^{\gamma}\mathcal{N}(\varphi)$ in $H^2$ it suffices to bound monomials $\norm{\partial^\gamma(\varphi^k)}_{H^2}$, $0\leq k\leq n$. Writing $\gamma=(\gamma_x,\gamma_y)$, we expand this to deduce that
\begin{equation}\label{eq:deriv_exp}
\begin{aligned}
 \norm{\partial^\gamma(\varphi^k)}_{H^2}&=\norm{\partial_x^{\gamma_x}\partial_y^{\gamma_y}(\varphi^k)}_{H^2}\leq C_0\sum_{\substack{i_1\leq\gamma_x,\\i_1'\leq\gamma_y}}\binom{\gamma_x}{i_1}\binom{\gamma_y}{i_1'}\norm{\partial_x^{\gamma_x-i_1}\partial_y^{\gamma_y-i_1'}\varphi}_{H^2}\norm{\partial_x^{i_1}\partial_y^{i_1'}(\varphi^{k-1})}_{H^2}\\
 &\leq C_0^k\sum_{\substack{i_1\leq\gamma_x,\\i_1'\leq\gamma_y}}\binom{\gamma_x}{i_1}\binom{\gamma_y}{ i_1'}\norm{\partial_x^{\gamma_x-i_1}\partial_y^{\gamma_y-i_1'}\varphi}_{H^2}\sum_{\substack{i_2\leq i_1,\\i_2'\leq i_1'}}\binom{i_1}{i_2}\binom{i_1'}{i_2'}\norm{\partial_x^{i_1-i_2}\partial_y^{i_1'-i_2'}\varphi}_{H^2}\ldots\\
 &\qquad\qquad\ldots \sum_{\substack{i_{k-1}\leq i_k,\\i_{k-1}'\leq i_{k}'}}\binom{i_{k-1}}{i_k}\binom{i_{k-1}'}{i_k'}\norm{\partial_x^{i_{k-1}-i_k}\partial_y^{i_{k-1}'-i_k'}\varphi}_{H^2}\norm{\partial_x^{i_k}\partial_y^{i_k'}\varphi}_{H^2}.
\end{aligned} 
\end{equation}
Now we note that for any $N\in\N$ we have
\begin{equation}
\begin{aligned}
 \sum_{m=0}^N\frac{1}{(1+m)^2(1+(N-m))^2}&\leq \frac{1}{(1+\lfloor\frac{N}{2}\rfloor)^2}\sum_{m=0}^{\lfloor\frac{N}{2}\rfloor}\frac{1}{(1+m)^2}+\frac{1}{(1+\lceil\frac{N}{2}\rceil)^2}\sum_{\lceil\frac{N}{2}\rceil}^{N}\frac{1}{(1+(N-m))^2}\\
 &\leq \frac{8}{(1+N)^2}\cdot\frac{\pi^2}{6}\leq \frac{15}{(1+N)^2},
\end{aligned} 
\end{equation}
so that appealing to the inductive hypothesis \eqref{eq:ind_hyp} we estimate that
\begin{equation}
\begin{aligned}
 &\sum_{\substack{i_{k-1}\leq i_k,\\i_{k-1}'\leq i_{k}'}}\binom{i_{k-1}}{i_k}\binom{i_{k-1}'}{i_k'}\norm{\partial_x^{i_{k-1}-i_k}\partial_y^{i_{k-1}'-i_k'}\varphi}_{H^2}\norm{\partial_x^{i_k}\partial_y^{i_k'}\varphi}_{H^2}\\
 &\leq C_\ast^2 R^{i_{k-1}}R^{i'_{k-1}}i_{k-1}!\, i_{k-1}'! \sum_{i_{k-1}\leq i_k}\frac{1}{(1+i_{k-1})^2(1+(i_k-i_{k-1}))^2}
 \sum_{i_{k-1}'\leq i_{k}'}\frac{1}{(1+i_{k-1}')^2(1+(i_k'-i_{k-1}'))^2}\\
 &\leq (15C_\ast)^2\cdot R^{i_{k-1}}R^{i'_{k-1}}i_{k-1}!\, i_{k-1}'!\cdot\frac{1}{1+(i_{k-1})^2}\frac{1}{1+(i_{k-1}')^2}.
\end{aligned}
\end{equation}
We iterate this in \eqref{eq:deriv_exp} to deduce that
\begin{equation}
\begin{aligned}
 \norm{\partial^\gamma(\varphi^k)}_{H^2}&\leq C_0^k\cdot (15C_\ast)^{2k}\cdot\frac{R^{\gamma_x}\gamma_x!}{(1+\gamma_x)^2}\frac{R^{\gamma_y}\gamma_y!}{(1+\gamma_y)^2}\leq 4 C_0^n\cdot (15C_\ast)^{2n}\frac{R^{\abs{\gamma}}\cdot\beta!}{(1+\beta_1)^2(1+\beta_2)^2}\\
 &\leq \frac{1}{nC_\mathscr{L} }\cdot\frac{R^{\abs{\beta}}\cdot\beta!}{(1+\beta_1)^2(1+\beta_2)^2},
\end{aligned} 
\end{equation}
where we used that $\abs{\gamma}+2=\abs{\alpha}$ and chose $R\geq2 C_0^{\frac{n}{2}}\cdot (15C_\ast)^n\cdot (n C_\mathscr{L} )^{\frac{1}{2}}$ large enough. This yields the induction claim \eqref{eq:ind_hyp} upon summation: 
\begin{equation}
 \Vert\partial^\beta\varphi\Vert_{\dot{H}^2}\leq C_\mathscr{L} \norm{\partial^{\gamma}\mathcal{N}(\varphi)}_{H^2}\leq nC_\mathscr{L} \sup_{0\leq k\leq n}\Vert\partial^\gamma(\varphi^k)\Vert_{H^2}\leq \frac{R^{\abs{\beta}}\cdot\beta!}{(1+\beta_1)^2(1+\beta_2)^2}.
\end{equation}

Now the conclusion follows swiftly: From \eqref{eq:ind_hyp} we deduce that
\begin{equation}
 \norm{\varphi}_{\dot{H}^k}\leq k\cdot \sup_{\abs{\alpha}=k-2}\norm{\partial^\alpha \varphi}_{H^2}\leq C_\ast\cdot R^k \cdot k!,
\end{equation}
so that for $0<\lambda<\frac{1}{2R}$ we have
\begin{equation}
 \norm{f}_{\mathcal{G}^\lambda}\leq C_\ast\sum_{k=0}^\infty\frac{\lambda^k}{k!}\norm{f}_{\dot{H}^{k}}\leq C_\ast \sum_{k=0}^\infty(\lambda R)^k\leq 2C_\ast.
\end{equation}
This concludes the proof.
\end{proof}

\subsection{Obstructions on the Square Torus}\label{ssec:Kolmo_obstruct}
Here we shed light on a piece in the puzzle towards understanding the local set of steady states around the Kolmogorov flow on $\T^2$. As mentioned already at the beginning of Section \ref{sec:Kolmo}, any nearby shear is trivially a stationary solution as well, as are flows of the form $\cos(y)+a\cos(x)+b\sin(x)$ for $a,b\in\R$ small enough.

In view of our previous results, it would thus be tempting to conjecture that the set of steady states can be identified in some sense with the kernel of $\Lin_K$ (as in the case of shear flows). That is, one could imagine that if one wanted to depart from Kolmogorov towards another stationary state, it should be possible to do so in the direction of any linearly neutral state. However, the following proposition shows that this is not the case: there are neutral directions that immediately take one outside the set of stationary states. To make this precise, let us denote by $\PP_K$ the projection onto $\ker\Lin_K$.

\begin{proposition}\label{prop:someflexT2}
If for some $\ell\in\N$, $\ell\geq 2$, 
\begin{equation}
 \frac{\PP_K(\Omega_\ast-\cos(y))}{\norm{\PP_K(\Omega_*-\cos(y))}_{L^2}}=\sin(\ell y)+\cos(x),
\end{equation} 
then there exists $\eps_0>0$ small so that if $\norm{\Omega_\ast-\cos(y)}_{H^6}=\eps<\eps_0$, then $\Omega_\ast$ is not a stationary solution to the $2d$ Euler equations. 
\end{proposition}
\begin{remark}
The coefficients of $\sin(\ell y)$ and $\cos(x)$ are not actually important in the proof, so as a result we obtain Proposition \ref{prop:main_someflex} from the introduction. As will be seen from the proof, the heuristic condition on $\PP_K(\Omega_\ast-\cos(y))$ is that its self-interaction \emph{not} lie in the range of $\mathcal{L}_K$. It seems that non-existence can be established in general under this condition, but note that lying in the range of $\mathcal{L}_K$ is not a sufficient condition for existence.
\end{remark}

\begin{proof}
Without loss of generality we treat the case $\ell=2$. By assumption, we can write 
\begin{equation}
 \omega_*:=\Omega_\ast-\cos(y)=a(\sin(2y)+\cos(x))+\widetilde\omega(x,y),
\end{equation}
where $a\in\R\setminus\{0\}$ and $\widetilde\omega\in \ker(\mathcal{L}_K)^\perp$ (i.e.\ $\widetilde\omega$ is orthogonal to all functions of $y$ only, as well as to $\sin(x)$ and $\cos(x)$). By assumption, we also know that
\begin{equation}
 \abs{a}^2+\norm{\widetilde\omega}_{H^6}^2\leq \eps^2.
\end{equation}
Now assume toward a contradiction that $\omega_\ast$ is a stationary solution to the $2d$ Euler equations. 
Then we have that $-\Lin_K\widetilde\omega=-\Lin_K\omega_\ast=u_\ast\cdot\nabla\omega_\ast$, where $u_\ast=\nabla^\perp\Delta^{-1}\omega_\ast$, and we compute this as 
\begin{equation}\label{eq:tilde}
\begin{aligned}
-\Lin_K\widetilde\omega=u_\ast\cdot\nabla\omega_\ast&=\left[a\begin{pmatrix}\frac{1}{2}\cos(2y)\\\sin(x)\end{pmatrix}+\widetilde{u}\right]\cdot\nabla\left[a\begin{pmatrix}-\sin(x)\\2\cos(2y)\end{pmatrix}+\widetilde{\omega}\right]\\
&=\frac{3}{2}a^2\sin(x)\cos(2y)+a\begin{pmatrix}\frac{1}{2}\cos(2y)\\\sin(x)\end{pmatrix}\cdot\nabla\widetilde\omega+\widetilde{u}\cdot a\begin{pmatrix}-\sin(x)\\2\cos(2y)\end{pmatrix}+\widetilde{u}\cdot\nabla\widetilde\omega.
\end{aligned}
\end{equation}
Here $\widetilde{u}=\nabla^\perp\Delta^{-1}\widetilde\omega$, which is average-free.

Next we note that a simple integration by parts gives that $\norm{\de_y(\sin(y)f)}_{L^2}\geq \frac{1}{2}\norm{f}_{L^2}$, with which it follows that
\begin{equation}
\begin{aligned}
 \norm{\widetilde\omega}_{L^2}&\leq 2\norm{\Lin_K\widetilde\omega}_{H^1}=2\norm{u_\ast\cdot\nabla\omega_\ast}_{H^1}\lesssim Ca^2+C\abs{a}\norm{\widetilde\omega}_{H^2}+C\norm{\widetilde\omega}_{H^2}^2\\
 &\leq Ca^2+C\abs{a}\norm{\widetilde\omega}_{L^2}^{1/2}\norm{\widetilde\omega}_{H^4}^{1/2}+C\norm{\widetilde\omega}_{L^2}\norm{\widetilde\omega}_{H^4}\\
 &\leq Ca^2 +C\eps^{1/2}\norm{\widetilde\omega}_{L^2}^{1/2}+C\eps\norm{\widetilde\omega}_{L^2},
\end{aligned} 
\end{equation}
where $C>0$ is a universal constant. Hence for $\eps>0$ sufficiently small we have
\begin{equation}
 \norm{\widetilde\omega}_{L^2}\leq Ca^2,
\end{equation}
implying that
\begin{equation}\label{eq:H3est}
 \norm{\widetilde\omega}_{H^3}^2\leq C\norm{\widetilde\omega}_{L^2}\norm{\widetilde\omega}_{H^6}\leq Ca^2\eps.
\end{equation}
On the other hand, evaluating \eqref{eq:tilde} at the point $(x,y)=(\pi/2,0)$, we use \eqref{eq:H3est} to find the bound
\begin{equation}
 0\geq\frac{3}{2}a^2-aC\norm{\widetilde\omega}_{W^{1,\infty}}-\norm{\widetilde\omega}_{W^{1,\infty}}^2\geq\frac{3}{2}a^2-Ca^2\eps^{1/2}-Ca^2\eps.
\end{equation}
For $\eps>0$ sufficiently small this implies that $a=0$, which is a contradiction.
\end{proof}

\subsection{Rigidity on Rectangular Tori}\label{ssec:Kolmo_rigid}
In this section we further highlight the role of the domain geometry in creating the global degeneracy we exploited in Theorem \ref{thm:main_Kolmo} to construct non-trivial stationary states near the Kolmogorov flow. We now show that this result is indeed special to the \emph{square} torus $\T^2$: we demonstrate that on \emph{rectangular} tori $\T^2_{\delta}=[0,2\pi \delta]\times [0,2\pi]$ with $\delta>0$ and $\delta\not\in\N$ (and periodic boundary conditions), all stationary, sufficiently regular $2d$ Euler flows near the Kolmogorov flow $U_K=(\sin(y),0)$ are purely shears. This is the content of our Theorem \ref{thm:main_rigid_intro}, restated here for convenience.

\begin{theorem*}[Rigidity on rectangular tori]\label{thm:main_rigid}
 Consider the stationary solution $U_K(x,y)=(\sin(y),0)$ on $\T^2_{\delta}$, $\delta>0$ with $\delta\not\in\N$, of the Euler equations \eqref{eq:NSEvel}. There exists $\eps_0>0$ (depending on $\delta$) such that if $U:\T^2_\delta\to\R^2$ is a further stationary solution to the Euler equations with 
 \begin{equation}\label{eq:ass_rigid}
  \norm{U-U_K}_{H^3}\leq \eps_0,
 \end{equation}
 then $U=U(y)$ is a shear flow.
\end{theorem*}

We observe that this rigidity is also witnessed at the level of the linearized operator $\Lin_K$: Since on $\T^2_\delta$ there exists $c_\delta>0$ such that 
\begin{equation}\label{eq:rect_bd}
 \norm{(1+\Delta_\delta^{-1})\de_x f}_{L^2}\geq c_\delta\norm{\de_x f}_{L^2},
\end{equation}
the kernel $\ker\Lin_K$ consists only of shears in $y$, a fact we rely on for our proof of this result. 

\begin{proof}[Proof of Theorem \ref{thm:main_rigid_intro}]
 We define $u:=U-U_K$, and let $\omega$ be the associated vorticity $\omega=\de_x u_2-\de_y u_1$, which has zero average on $\T^2_\delta$ and satisfies $\norm{\omega}_{H^2}\lesssim\eps_0$ by assumption \eqref{eq:ass_rigid}. Since $U$ is a stationary Euler solution, we have that $\omega$ satisfies the equation
 \begin{equation}
  \Lin_K \omega+u\cdot\nabla\omega=0.
 \end{equation}
 Writing $\Delta_\delta$ for the  Laplacian on $\T^2_\delta$, we have that $\Lin_K=\sin(y)(1+\Delta_\delta^{-1})\de_x$, and thus 
 \begin{equation}
  \int \de_y\left(\sin(y) (1+\Delta_\delta^{-1})\de_x \omega\right)\cos(y)(1+\Delta_\delta^{-1})\de_x \omega=-\int \partial_y(u\cdot\nabla \omega)\cos(y)(1+\Delta_\delta^{-1})\de_x\omega.
 \end{equation}
 Together with the identity 
 \begin{equation}
  \int \de_y \left(\sin(y)f\right)\cos(y)f=\frac{1}{2}\int f^2
 \end{equation}
 and \eqref{eq:rect_bd}, it then follows that for some universal constant $C>0$ we have
 \begin{equation}\label{eq:ineq_rigid}
 \begin{aligned}
  2c_{\delta}\norm{\de_x\omega}_{L^2}^2&\leq 2\norm{(1+\Delta_\delta^{-1})\,\de_x\omega}_{L^2}^2\leq \abs{\int \partial_y(u\cdot\nabla \omega)\cos(y)(1+\Delta_\delta^{-1})\de_x\omega}\\
  &\leq \norm{\de_yu_1}_{L^\infty}\norm{\de_x\omega}_{L^2}^2+\norm{\de_yu_2}_{L^4}\norm{\de_y\omega}_{L^4}\norm{\de_x\omega}_{L^2}\\
  &\qquad +\norm{u_2}_{L^\infty}\norm{\de_{yy}\omega}_{L^2}\norm{\de_x\omega}_{L^2}+\abs{\int u_1\de_{xy}\omega\cos(y)(1+\Delta_\delta^{-1})\de_x\omega}\\
  &\leq C\norm{\de_x\omega}_{L^2}^2\norm{\omega}_{H^2},
 \end{aligned} 
 \end{equation}
 where we used that $\norm{u_2}_{L^\infty}=\norm{\de_x\Delta_\delta^{-1}\omega}_{L^\infty}\lesssim \norm{\de_x\omega}_{L^2}$ since $\omega$ is average free, and that by integration by parts there holds
 \begin{equation}
  \abs{\int u_1\de_{xy}\omega\cos(y)(1+\Delta_\delta^{-1})\de_x\omega}\leq \norm{\de_x\omega}_{L^2}^2\left[\norm{\de_y u_1}_{L^\infty}+2\norm{u_1}_{L^\infty}\right].
 \end{equation}
In conclusion, if we assume that $\eps_0<\frac{2c_\delta}{C}$, equation \eqref{eq:ineq_rigid} can only hold if $\de_x\omega=0$, i.e.\ if $\omega$ (and thus also $U$) is a pure shear flow.
\end{proof}

\section{No Threshold for Enhanced Dissipation near Bar States on $\T^2$}\label{sec:Bar}
We now turn to the closely related question of the dynamical behavior of solutions to the Navier-Stokes equations near a \emph{bar state} $\Omega_{bar}(t,y)=\e^{-\nu t}\cos(y)$. One verifies directly that $\Omega_{bar}(t,y)$ satisfies the two dimensional Navier-Stokes equations \eqref{eq:NSEvort} on $\T^2$.

Now consider a solution $f(t)$ of the linearized equation near a bar state,
\begin{equation}\label{eq:Bar_lin2}
 \de_t f+\e^{-\nu t}\Lin_K f=\nu\Delta f.
\end{equation}
One sees directly that the projection of $f$ onto $\ker\Lin_K=\{\cos(x),\sin(x)\}\cup\{f\in L^2:\de_x f\equiv 0\}$ simply obeys a heat equation, so no decay beyond the natural time scale $O(\nu^{-1})$ can hold. However, once one projects away from $\ker\Lin_K$ to $\mathcal{D}:=(\ker\Lin_K)^\perp$, (inviscid) advection and diffusion conspire to create an enhanced rate of dissipation \cite{WZ19,WZZ20,SMM19}, such that solutions $f$ to \eqref{eq:Bar_lin2} satisfy
\begin{equation}
 \norm{\PP_\D f(t)}_{L^2}\lesssim \e^{-c_1\nu^{1/2}t}\norm{\PP_\D f(0)}_{L^2}, \qquad \forall t\leq \frac{\tau}{\nu},
\end{equation}
where $c_1>0$ is some universal constant and $\tau>0$ is arbitrary. 

However, as discussed in the introduction, our Theorem \ref{thm:main_Bar} demonstrates that there \emph{cannot be any threshold} below which such $L^2$ decay also holds in the nonlinear Navier-Stokes problem near the bar states on $\T^2$. This follows from the fact that there exist initial data $\Omega_\eps$, arbitrarily close to those of the bar states, that do not lead to decay before the diffusive time scale $O(\nu^{-1})$ is reached. Our proof establishes this as follows.

\subsection{Proof of Theorem \ref{thm:main_Bar}}
Let $\Psi_\eps$ be a stationary stream function for the Euler equations as in Theorem \ref{thm:main_Kolmo} (or Proposition \ref{prop:KolmoH2}). Recall from Lemma \ref{lem:psi_eps_exp} that with $g(y):=c_0\cos(3y)-c_1\cos(5y)$ and $\phi_\eps\in H^2(\T^2)$ it can then be written as
\begin{equation}
 \Psi_\eps=\cos(y)+\eps\cos(x)+\eps g(y)+\eps^2\phi_\eps(x,y),
\end{equation}
with moreover $\int_{\T^2}\Psi_\eps=0$ (as follows directly from construction as a solution to \eqref{eq:stat_sol}). One computes directly that
\begin{equation}
 U_\eps=\nabla^\perp\Psi_\eps=\begin{pmatrix}\sin(y)-\eps g'(y)\\ -\eps\sin(x)\end{pmatrix}+O(\eps^2),\quad \Omega_\eps=-\cos(y)-\eps\cos(x)+\eps g''(y)+O(\eps^2).
\end{equation}
Dropping for simplicity of notation the subscript $\eps$, we define now the heat flow $\Omega^h$ of $\Omega_\eps$ as
\begin{equation}
 \Omega^h(t):=\e^{\nu t\Delta}\Omega,\quad U^h(t):=\e^{\nu t\Delta}U,
\end{equation}
which solves $\de_t\Omega^h=\nu\Delta\Omega^h$. For future use we note that also
\begin{equation}
 \norm{U^h\cdot\nabla\Omega^h}_{L^2}\leq C_1 \e^{-2\nu t}\eps^2.
\end{equation}
Next we compare this with the solution of the Navier-Stokes equations starting at $\Omega_\eps$: let $\Omega^\nu$ solve
\begin{equation}\label{eq:NS_eps}
 \de_t\Omega^\nu+U^\nu\cdot\nabla\Omega^\nu=\nu\Delta\Omega^\nu,\quad \Omega^\nu(0)=\Omega_\eps.
\end{equation}
The difference $\Omega^\nu-\Omega^h$ then solves
\begin{equation}
 \de_t(\Omega^\nu-\Omega^h)+U^\nu\cdot\nabla\Omega^\nu=\nu\Delta(\Omega^\nu-\Omega^h),
\end{equation}
and we note that $\int_{\T^2}\Omega^\nu-\Omega^h=0$. A standard energy estimate then gives
\begin{equation}
 \frac{1}{2}\ddt\norm{\Omega^\nu-\Omega^h}_{L^2}^2+\nu\norm{\nabla(\Omega^\nu-\Omega^h)}_{L^2}^2=\abs{\ip{U^\nu\cdot\nabla\Omega^\nu,\Omega^\nu-\Omega^h}_{L^2}}.
\end{equation}
Since
\begin{equation}
 \ip{U^\nu\cdot\nabla\Omega^\nu,\Omega^\nu-\Omega^h}=\ip{(U^\nu-U^h)\cdot\nabla\Omega^h,\Omega^\nu-\Omega^h}+\ip{U^h\cdot\nabla\Omega^h,\Omega^\nu-\Omega^h}
\end{equation}
and since $\Omega^\nu-\Omega^h$ is mean-free we have
\begin{equation}
 \abs{\ip{U^\nu\cdot\nabla\Omega^\nu,\Omega^\nu-\Omega^h}_{L^2}}\leq \norm{\nabla\Omega^h}_{L^\infty}\norm{\Omega^\nu-\Omega^h}_{L^2}^2+C\e^{-2\nu t}\eps^2\norm{\Omega^\nu-\Omega^h}_{L^2}.
\end{equation}
By Gr\"onwall's Lemma it thus follows that
\begin{equation}
 \norm{\Omega^\nu(t)-\Omega^h(t)}_{L^2}\leq C_1 t\eps^2\exp\left(\int_0^t \norm{\nabla\Omega^h(s)}_{L^\infty}\dd s\right)\leq C_1 t\eps^2\exp\left(t\norm{\nabla\Omega_\eps}_{L^\infty}\right).
\end{equation}
Since there exists $C_2>0$ such that for all $\eps>0$ one has $\norm{\nabla\Omega_\eps}\leq C_2$, we may choose $\eps_0=\frac{\nu \e^{-\frac{C_2}{\nu}}}{100 C_1}$, which implies that for $t\in[0,\frac{1}{\nu}]$ we have
\begin{equation}
 \norm{\Omega^\nu(t)-\Omega^h(t)}_{L^2}\leq \frac{\eps}{100}.
\end{equation}
This shows that for $t\in[0,\frac{1}{\nu}]$ we have
\begin{equation}\label{eq:NS_expanded}
\begin{aligned}
 \Omega^\nu&=-\alpha(t)\e^{-\nu t}\cos(y)-\eps\beta(t)\e^{-\nu t}\cos(x)-9c_0\eps\gamma(t)\e^{-9\nu t}\cos(3y)+25c_1\eps\delta(t)\e^{-25\nu t}\cos(5y)\\
 &\qquad+\eps^2 H(t,x,y),
\end{aligned} 
\end{equation}
with a remainder $H\in C^1_tH^2_{x,y}([0,\frac{1}{\nu}]\times\T^2)$ and for differentiable maps $\alpha:[0,\frac{1}{\nu}]\to[1-\frac{\eps}{100},1+\frac{\eps}{100}]$ and $\beta,\gamma,\delta:[0,\frac{1}{\nu}]\to [\frac{99}{100},\frac{101}{100}]$. Expanding now the equation \eqref{eq:NS_eps} for $\Omega^\nu$ in powers of $\eps$ up to first order shows that in fact $\alpha'=\beta'=\gamma'=\delta'=0$, i.e.\ $\alpha=\beta=\gamma=\delta=1$.

Expanding at order $\eps^2$, we see that for some $K_i\neq 0$, $1\leq i\leq 2$, we have
\begin{equation}\label{eq:eps2_ver1}
 \PP_\D \left[\frac{\dd^2}{\dd\eps^2}|_{\eps=0}(U^\nu\cdot\nabla\Omega^\nu)\right]=K_1\e^{-10\nu t}\sin(x)\sin(3y)+K_2 \e^{-26\nu t}\sin(x)\sin(5y)+\PP_\D(\sin(y)\cdot H).
\end{equation}
Note now that in the last term, the only way to create a mode $\sin(x)\sin(3y)$ or $\sin(x)\sin(5y)$ is by having $\cos(y)$ (the zero order part of $\Omega^\nu$) interact with an element of $\PP_\D H$.
Assuming that enhanced dissipation happens, however, we have that $\norm{\PP_\D H(t)}_{L^2}\ll\eps$ for $t\in[\frac{1}{2\nu},\frac{1}{\nu}]$, so that on this time interval in fact we can conclude from \eqref{eq:eps2_ver1} that
\begin{equation}\label{eq:eps2_ver2}
\begin{aligned}
 &\ip{\PP_\D \left[\frac{\dd^2}{\dd\eps^2}|_{\eps=0}(U^\nu\cdot\nabla\Omega^\nu)\right],\sin(x)\sin(3y)}=\pi^2 K_1\e^{-10\nu t},\\
 &\ip{\PP_\D \left[\frac{\dd^2}{\dd\eps^2}|_{\eps=0}(U^\nu\cdot\nabla\Omega^\nu)\right],\sin(x)\sin(5y)}=\pi^2 K_2 \e^{-26\nu t}.
\end{aligned}
\end{equation}
However, this contradicts the assumption of enhanced dissipation of $\PP_\D\Omega^\nu$, and thus concludes the proof of Theorem \ref{thm:main_Bar}.

\section{Rigidity near Poiseuille flow}\label{sec:Pois}
In this section we prove Theorem \ref{thm:main_Pois}, which asserts that the only (sufficiently regular) traveling wave solutions near the Poiseuille flow are in fact shears. The idea is as follows: First we split a given traveling wave into a shear and non-shear part. By assumption, the shear part is close to the Poiseuille flow, and we show in Proposition \ref{prop:lowerBound} that the linearized operator near it satisfies a strong coercivity estimate. This can then be employed to show that if the regularity is sufficiently high ($H^{5+}$), then the non-shear part has to vanish.

The following result gives the announced strong coercivity estimate for the linearized operator
\begin{equation}\label{eq:linearized}
 \Lin_V:=V(y)\de_x - V''(y)\Delta^{-1}\de_x,
\end{equation} 
around a shear flow $(V(y),0)$ near Poiseuille flow $U_P=(y^2,0)$.
\begin{proposition}\label{prop:lowerBound}
There exist constants $c_1, \eps_1>0$ with the following property:  let $\psi \in H^3$ be such that
\begin{equation}\label{eq:x-ave}
 \int_\T \psi(x,y)\dd x=0.
\end{equation}
If  $\eps\in(0,\eps_1)$ and $V\in W^{5,\infty}([-1,1])$ is such that 
\begin{equation}\label{eq:CloseToP}
 \norm{V'-2y}_{W^{4,\infty}}<\eps,
\end{equation}
then
\begin{equation}\label{eq:P2}
 \norm{\Lin_V\omega}_{\dot{H}^1}\geq c_1\frac{\norm{\de_{x}\nabla\psi}_{L^2}^2+\norm{V'\de_{x}\omega}_{L^2}^2}{\norm{\omega}_{\dot{H}^1}}.
\end{equation}
\end{proposition}

\begin{proof}
Since the linear operator $\Lin_V$ decouples in the $x$-frequency, we expand $\omega$ (and $\psi$) as a Fourier series in the $x$ variable, namely
\begin{equation}
\omega(t,x,y)=\sum_{\ell\in \mathbb{Z}} a_\ell(t,y)\e^{i \ell x}, \qquad a_\ell(t,y)=\frac{1}{2\pi}\int_\T\omega(t,x,y)\e^{-i \ell x}\dd x.
\end{equation} 
For $k\in\bbN_0$ we set
\begin{equation}
 \omega_k(t,x,y):=\sum_{|\ell|=k} a_k(t,y)\e^{i\ell x}.
\end{equation}
Thanks to \eqref{eq:x-ave}, we may express $\omega=\sum_{k\in\bbN}\omega_k(t,x,y)$ as a sum of \emph{real-valued} functions $\omega_k$ that are localized in $x$-frequency on a single band $\pm k$, $k\in\bbN$.

Define now
\begin{equation}\label{eq:defAU}
A_{V}(\omega):=\l V'\de_x \Lin_V\omega,\de_y \omega\r+\l V'\de_y\Lin_V\omega, \de_x\omega\r.
\end{equation}
Recalling that $\psi$ satisfies homogeneous Dirichlet boundary conditions in $y$, a direct computation shows that 
\begin{equation}
\begin{aligned}
A_{V}(\omega)
&=k^2\left[\| V'\omega\|^2_{L^2}+\l V''V'\psi, \de_y \omega\r-\l V''V'\de_{y}\psi,\omega\r-\l V'''V'\psi,\omega\r\right]\\
&=k^2\left[\| V'\omega\|^2_{L^2}+\l V''V'\psi, \de_y \Delta\psi\r-\l V''V'\de_{y}\psi,\Delta\psi\r-\l V'''V'\psi,\Delta\psi\r\right]\\
&=k^2\big[\| V'\omega\|^2_{L^2}-\l (V''V')'\psi, \de_{yy}\psi\r+\l (V''V')'\de_{y}\psi,\de_y\psi\r-\l V'''V'\psi,\Delta\psi\r\big]\\
&=k^2\big[\| V'\omega\|^2_{L^2}+\l (V''V')''\psi, \de_{y}\psi\r+2\l (V''V')'\de_y\psi, \de_{y}\psi\r+\l (V'''V')'\psi,\de_y\psi\r\\
&\qquad\qquad +\l V'''V'\nabla\psi,\nabla\psi\r\big]\\
&=k^2\Big[\| V'\omega\|^2_{L^2}+2\| V''\de_y\psi\|^2_{L^2}-\frac12\l \left[(V''V')'''+(V'''V')''-2k^2 V'''V'\right]\psi, \psi\r\\
&\qquad\qquad  +3\l V'''V'\de_y\psi, \de_{y}\psi\r\Big].
\end{aligned}
\end{equation}
Using \eqref{eq:CloseToP}, we have that in particular $V''\geq 1$. Therefore, since $k^2\geq 1$, for some $c_2\geq1$ we find that 
\begin{equation}\label{eq:comp1}
A_{V}(\omega)
\geq k^2\left[\| V'\omega\|^2_{L^2}+2\| \de_y\psi\|^2_{L^2} -c_2\eps k^2 \|\psi\|_{L^2}^2 -c_2\eps\| \de_{y}\psi\|_{L^2}^2\right].
\end{equation}
Now observe that
\begin{equation}
 \l \de_{xy}\psi, V'\de_x\omega \r=\l \de_{xy}\psi, V'\de_{xyy}\psi \r+\l \de_{xy}\psi, V'\de_{xxx}\psi \r=-\frac12\l V''\de_{xy}\psi,\de_{xy}\psi\r+\frac12\l V''\de_{xx}\psi,\de_{xx}\psi\r
\end{equation}
so that
\begin{equation}
\l V''\de_{xx}\psi,\de_{xx}\psi\r-\l V''\de_{xy}\psi,\de_{xy}\psi\r=2 \l \de_{xy}\psi, V'\de_x\omega \r\leq   \|\de_{xy}\psi\|^2+\|V'\de_{x}\omega\|^2.
\end{equation}
Since $1\leq V''\leq 3$, it follows that
\begin{equation}\label{eq:inequal}
\|\de_{xx}\psi \|^2_{L^2}\leq  4\|\de_{xy}\psi\|^2+\|V'\de_{x}\omega\|^2,
\end{equation}
or, equivalently,
\begin{equation}
k^2\|\psi \|^2_{L^2}\leq  4\|\de_{y}\psi\|^2+\|V'\omega\|^2,
\end{equation}
In particular, from \eqref{eq:comp1} we deduce that
\begin{equation}
A_{V}(\omega)\geq k^2\left[(1-c_2\eps)\| V'\omega\|^2_{L^2}+(2-5c_2\eps)\| \de_y\psi\|^2_{L^2} \right].
\end{equation}
Taking $\eps<\eps_1\leq 1/(5c_2)$ implies the lower bound
\begin{equation}\label{eq:lower}
A_{V}(\omega)\geq \frac{1}{2}\left[\| V'\de_x\omega\|^2_{L^2}+\| \de_{xy}\psi\|^2_{L^2} \right].
\end{equation}
A further use of \eqref{eq:inequal} then gives
\begin{equation}\label{eq:lower2}
A_{V}(\omega)\geq \frac{1}{16}\left[\| V'\de_x\omega\|^2_{L^2}+\| \de_{x}\nabla\psi\|^2_{L^2} \right].
\end{equation}
On the other hand, from the definition of $A_V(\omega)$ in \eqref{eq:defAU}, we have the upper bound
\begin{equation}\label{eq:upper}
A_{V}(\omega)\lesssim \| \de_x\Lin_V\omega\|_{L^2} \|\de_{y} \omega\|_{L^2}+ \| \de_y\Lin_V\omega\|_{L^2} \|\de_{x} \omega\|_{L^2}\lesssim
\| \Lin_V\omega\|_{\dot{H}^1} \| \omega\|_{\dot{H}^1}.
\end{equation}
Putting together \eqref{eq:lower2} and \eqref{eq:upper}, we obtain \eqref{eq:P2}  and we conclude the proof of the proposition. 
\end{proof}

\subsection{Proof of Theorem \ref{thm:main_Pois}}
Let us now consider a general traveling wave solution to the $2d$ Euler equations. Such a solution is necessarily of the form
\begin{equation}\label{eq:travEul1}
U(x-ct,y)=
\begin{pmatrix}
U_1(x-ct,y)\\
U_2(x-ct,y)
\end{pmatrix},
\end{equation}
for some $c\in \R$, and satisfies
\begin{equation}\label{eq:travEul2}
(U_1-c)\de_x\Omega+U_2\de_y\Omega=0,\qquad U=\nabla^\perp \Psi, \qquad \Delta \Psi =\Omega.
\end{equation}
We consider now its deviation from the Poiseuille flow, defining $\widetilde\psi$ as 
\begin{equation}\label{eq:defint}
\widetilde{\psi}(x,y):= \Psi(x,y)-\Psi_P(y),\qquad \psi(x,y):= \widetilde{\psi}(x,y)-\int_\T \widetilde{\psi}(x,y)\dd x,
\end{equation}
and notice that
\begin{equation}\label{eq:psimeanfree}
\int_\T \psi(x,y)\dd x=0.
\end{equation}
Accordingly, we set
\begin{equation}
u=\nabla^\perp \psi, \qquad \omega=\nabla^\perp\cdot u, \qquad \widetilde{u}=\nabla^\perp \widetilde{\psi}, \qquad \widetilde{\omega}=\nabla^\perp\cdot \widetilde{u}.
\end{equation}
From this we consider the shear part $(V(y),0)$, where
\begin{equation}
 V(y):=y^2-c-\de_y\int_\T \widetilde{\psi}(x,y)\dd x,\qquad  V'(y)=2y-\int_\T \widetilde{\omega}(x,y)\dd x.
\end{equation}
In light of the smallness assumption \eqref{eq:poisClose}, we obtain in particular that
\begin{equation}\label{eq:omegasmall}
\| \omega\|_{\dot{H}^5}\leq2\eps_0.
\end{equation}
From the equation \eqref{eq:travEul2} for the traveling wave and the definition of the linearized operator
$\Lin_V$ in \eqref{eq:linearized}, it follows that 
\begin{equation}
\Lin_V\omega=-u\cdot\nabla \omega, \qquad \textnormal{where }\Lin_V\omega=V\de_x\omega - V''\de_x\psi.
\end{equation}
Moreover, by virtue of the assumption \eqref{eq:poisClose} of proximity of the traveling wave to the Poiseuille flow, $V$ satisfies \eqref{eq:CloseToP} for $\eps_0\leq \eps_1$ as given by Proposition \ref{prop:lowerBound}. Hence Proposition \ref{prop:lowerBound} implies that
\begin{equation}
\|\de_{x}\psi\|_{\dot{H}^1}^2\leq \frac{1}{c_1} \|\omega\|_{\dot{H}^1}\|\Lin_V\omega\|_{\dot{H}^1}=\frac{1}{c_1} \|\omega\|_{\dot{H}^1}\|u\cdot \nabla \omega\|_{\dot{H}^1}.
\end{equation}
By interpolation, standard estimates and  \eqref{eq:psimeanfree} we have
\begin{align}
\|u\cdot \nabla \omega\|_{\dot{H}^1}
&\lesssim \|\nabla u\|_{L^\infty} \| \omega\|_{\dot{H}^1}+\| u\|_{L^\infty}  \|\omega\|_{\dot{H}^2}\notag\\
&\lesssim \|\nabla u\|_{L^2}^{1/2} \|\nabla u\|_{\dot{H}^2}^{1/2}  \| \omega\|_{\dot{H}^1}+\| u\|_{L^2}^{1/2} \| u\|_{\dot{H}^2}^{1/2}  \|\omega\|_{\dot{H}^2}\notag\\
%&\lesssim \|\Delta\psi \|_{L^2}^{1/2} \|\Delta\psi\|_{\dot{H}^2}^{1/2}  \| \Delta\psi\|_{\dot{H}^1}+\| \nabla \psi\|_{L^2}^{1/2} \| \nabla \psi \|_{\dot{H}^2}^{1/2}  \|\Delta\psi\|_{\dot{H}^2}\notag\\
%&\lesssim \|\psi \|_{\dot{H}^2}^{1/2} \|\psi\|_{\dot{H}^4}^{1/2}  \| \psi\|_{\dot{H}^3}+\|  \psi\|_{\dot{H}^1}^{1/2} \|  \psi \|_{\dot{H}^3}^{1/2}  \|\psi\|_{\dot{H}^4}\notag\\
%&\lesssim \|\psi \|_{\dot{H}^1}^{1/3}  \|\psi \|_{\dot{H}^4}^{1/6} \|\psi\|_{\dot{H}^4}^{1/2}  \| \psi\|^{1/3}_{\dot{H}^1} \| \psi\|^{2/3}_{\dot{H}^4}+\|  \psi\|_{\dot{H}^1}^{1/2} \|  \psi \|_{\dot{H}^1}^{1/6} \|  \psi \|_{\dot{H}^4}^{1/3}  \|\psi\|_{\dot{H}^4}\notag\\
&\lesssim \|\psi \|_{\dot{H}^1}^{2/3}  \|\psi \|_{\dot{H}^4}^{4/3}+\|  \psi\|_{\dot{H}^1}^{2/3} \|  \psi \|_{\dot{H}^4}^{4/3} \notag\\
&\lesssim \| \de_x \psi\|_{\dot{H}^1}^{2/3} \|  \psi \|_{\dot{H}^4}^{4/3}.
\end{align}
Therefore  it follows that 
\begin{align}\label{eq:1stbd}
\|\de_{x}\psi\|_{\dot{H}^1}^2\lesssim  \|\omega\|_{\dot{H}^1} \| \de_x \psi\|_{\dot{H}^1}^{2/3} \|  \psi \|_{\dot{H}^4}^{4/3}
\lesssim    \|\psi\|_{\dot{H}^3} \| \de_x \psi\|_{\dot{H}^1}^{2/3} \|  \psi \|_{\dot{H}^4}^{4/3}
\lesssim    \| \de_x \psi\|_{\dot{H}^1} \|  \psi \|_{\dot{H}^4}^{2}.
\end{align}
Finally, we interpolate once more and use that $\psi$ has zero $x$-average to deduce that
\begin{align}
\|\psi\|_{\dot{H}^4}\lesssim\|\psi\|_{\dot{H}^1}^{1/2}\|\psi\|_{\dot{H}^7}^{1/2}\lesssim\|\de_x\psi\|_{\dot{H}^1}^{1/2}\|\omega\|_{\dot{H}^5}^{1/2}.
\end{align}
Combined with \eqref{eq:1stbd}, this shows that there exists a constant $c_3\geq 1$ such that
\begin{equation}
\|\de_{x}\psi\|_{\dot{H}^1}\leq  c_3 \|\de_{x}\psi\|_{\dot{H}^1}\|\omega\|_{\dot{H}^{5}}.
\end{equation}
In view of \eqref{eq:omegasmall}, if we choose $\eps_0=\min\{1/(2c_3),\eps_1\}$, 
the only way the above inequality is satisfied is if $\de_x \psi\equiv 0$, that is,
$\psi\equiv 0$. In this case, from the relations \eqref{eq:defint} we obtain that $\Psi$ is only a function of $y$, and therefore the associated velocity
is a pure shear. The proof of Theorem \ref{thm:main_Pois} is over.

\subsection*{Acknowledgments}
The authors thank Theo Drivas and Toan T.\ Nguyen for inspiring conversations.

M.\ Coti Zelati acknowledges funding from the Royal Society through a University Research Fellowship (URF\textbackslash R1\textbackslash 191492). T.\ M.\ Elgindi acknowledges funding from NSF DMS-1817134 and NSF DMS-1945669.

\bibliographystyle{abbrv}
\bibliography{CZEW_StationaryStates}

\end{document}